\documentclass[10pt]{article}
\usepackage{amsfonts}
\usepackage{amsmath}
\usepackage{amsthm}
\usepackage{epsfig}
\usepackage{url}
\usepackage{graphicx}
\usepackage{epstopdf}
\usepackage{comment}
\usepackage{changepage}

\usepackage{epstopdf}
\usepackage{subfigure}
\usepackage{comment}

\theoremstyle{plain}
\newtheorem{theorem}{Theorem}[section]
\newtheorem{corollary}[theorem]{Corollary}
\newtheorem{lemma}[theorem]{Lemma}

\theoremstyle{definition}

\theoremstyle{remark}

\def\zbar{\overline{z}}

\def\mR{\mathbb{R}}
\def\J1inv{\widehat{J}_0}

\begin{document}



\title{\textit{Classification of critical sets and their images for quadratic maps of the plane}}

\author{Chia-Hsing Nien\\
Department of Financial and Computational Mathematics\\
Providence University\\
Taichung City 43301 Taiwan\\
chnien@pu.edu.tw\\ \ \\
Bruce B. Peckham\\
Department of Mathematics and Statistics\\
University of Minnesota Duluth, Duluth, MN 55812, USA\\
  bpeckham@d.umn.edu\\ \ \\
Richard P. McGehee\\
School of Mathematics\\ University of Minnesota, Minneapolis, MN 55455, USA\\
mcgehee@umn.edu}

\maketitle



\pagestyle{myheadings}
\thispagestyle{plain}
\markboth{\sc Nien, Peckham \& McGeghee}
{\sc Critical sets and their images for quadratic maps of the plane}

\begin{abstract}
We provide a complete classification of the critical sets and their images for quadratic maps of the real plane.
Critical sets are always conic sections, which provides a starting point for the classification.
The generic cases, maps whose critical sets are either ellipses or hyperbolas, was published in Delgado, {\it et al.} \cite{DGRRV}.
This work completes the classification by including all the nongeneric cases:
 the empty set, a single point, a single line, a parabola, two parallel lines, two intersecting lines, or the whole plane.
We describe all possible images for each critical set case and illustrate the geometry of representative maps for each case.

\end{abstract}

\noindent
{\bf Keywords:}
\texttt Quadratic maps, maps of the real plane, singularities, critical sets, geometric equivalence


\section{Introduction}

We are interested in studying the dynamics of quadratic maps of the real plane $\mR^2$.
Even in this restricted class, there is a huge collection of map behaviors and bifurcations that is still unknown.
This may seem surprising at first, but becomes less so when one reflects on the depth of the dynamics in the study of special cases such as the complex quadratic family $z^2+c$ \cite{Devaney89book, Milnorbook} or the Henon map \cite{Henon}.
Many researchers have investigated spectacular behavior of other subfamilies of quadratic maps.
Most such studies, with good reason, have restricted their studies to families with only one or two parameters.
For representative studies, see 
\cite{AGMbook, ACHM,  FKP, GuM1, GuM2, Lorenz, MiraGBCbook, RRVil, RRViv, RSV}.
Research more in the spirit of our classification approach (described below), but still for a restricted set of quadratic maps, includes \cite{BGVRR}, where the authors study quadratic maps with no fixed points, and \cite{Nien98} where maps with bounded critical sets (points or ellipses) are studied.

In this paper, we take a complementary approach.
Rather than a fairly complete understanding of the dynamics of a small family of quadratic maps, we obtain a complete understanding for {\it all} quadratic maps, but only for their critical sets and images.
The fact that the behavior of critical sets can completely determine the dynamics of a map is well-established in one-dimensional real and complex dynamics \cite{Devaney89book, dMvSbook, Milnorbook}.
Such a complete classification based on critical orbits is not established for noninvertible maps of the real plane,
but critical orbits are still clearly important \cite{MiraNara,FGKMM, MiraMG}.
The study of critical orbits in the real plane setting is more difficult in part because the critical sets are typically curves rather than the isolated points
which typically occur in one-dimensional real or complex maps. 
Here we consider only the critical set and its first image, rather than the full orbits of critical sets. 
With this coarser classification we gain the ability to attain a complete classification of all quadratic maps of the plane.
From this perspective, this paper is more of a global singularity classification than a dynamical classification.

{\bf Problem history and significance.}
Much of the work presented in this paper appeared in 1997 as a chapter in the Ph.~D.~thesis \cite{Nienthesis} of author CHN under the direction of author RPM but was not peer reviewed at that time.
Nien classified the images of critical sets for quadratic maps in cases where the critical set was given by a quadratic in two variables having at least one nonzero quadratic term.
(In Theorem \ref{th-J0J1} below, this includes cases 3,4,5,6, and 7 of cases 1-9.)
The only published work of which we are aware which overlaps significantly with the current paper is 2013 paper by Delgado, et al.
\cite{DGRRV}
where the authors consider the two generic cases (3 and 4 in Theorem \ref{th-J0J1}) where the critical set is an ellipse or a hyperbola.
In that work the authors obtain a stronger result for these two cases than stated in our Theorem \ref{th-J0J1}.
They obtain a full geometric equivalence classification of these maps, not just a classification of the critical points and their images as we provide in this paper.
See also \cite{GRRV}, a 2005 preprint version
of \cite{DGRRV}.
In the current paper we complete the classification to include all cases for quadratic maps of the plane.

We emphasize that while the `nongeneric' cases (in the coefficient space of the quadratic maps) are unlikely for random choices of coefficients, they appear quite frequently in the literature of quadratic maps.
For example, the class that many consder the simplest noninvertible quadratic maps are the so-called $Z_0$-$Z_2$ maps, which all have a line as a critical set and a critical image \cite{AGMbook, GuM1, GuM2, MiraGBCbook}.  Examples with a parabola as a critical set include \cite{FKP, Lorenz}. A critical set of parallel lines is included in \cite{MiraGBCbook}. The delayed logistic example in \cite{ACHM} has a critical line but its image is a point. 
Complex quadratic maps have a critical point.
Of course, nongeneric cases are extremely important in the bifurcation theory associated with analyzing families of quadratic maps. Nongeneric cases necessarily appear in transitions between generic cases.

{\bf Background.} 
We start with the most general quadratic map of the real plane, $F:\mR^2 \rightarrow \mR^2$, with twelve coefficient parameters.

\begin{align}
\nonumber
F(x,y)=(a_0x^2&+a_1xy+a_2y^2+a_3x+a_4y+a_5,\\ &b_0x^2+b_1xy+b_2y^2+b_3x+b_4y+b_5)
\label{eq-genquad}
\end{align}

\noindent
We assume that at least one of the six quadratic coefficients is nonzero, so we exclude affine maps of the plane;
the dynamics of affine planar maps is well-known.
We define the critical set $J_0^F$ (or just $J_0$ when the map is clear) by
\begin{eqnarray}
        J_0^F = \left\{(x,y)\in \mR^2 \,|\;
        \det(\mathit{DF}\left(x,y\right)) =0 \right\}.
\end{eqnarray}
The image $J_1 = F(J_0)$ is called the {\it critical image} or {\em critical locus}.  

Generally, noninvertibility gives rise to regions with
different numbers of preimages.
The critical locus $J_1$ divides the
phase plane $\mR^2$ into regions with a constant number of pre-images.
These regions are usually labeled by $Z_k$, where $k$ is the number of
pre-images in that region \cite{MiraMG}.  
The map $f$ folds the phase plane along smooth
curves of $J_0$.
Its image $J_1$ is a smooth curve except possibly at isolated cusp points.
Therefore, generically, the number of
pre-images differs locally by two on either side of $J_1$ \cite{Arnold, GG},
and as one moves from one region into the next
by crossing a fold curve the number of pre-images changes from $k$ to $k \pm 2$.
This principle is violated in several nongeneric cases, where two branches of $J_0$ curves are mapped to the same curve in $J_1$, creating a change in preimages by four instead of two (cases 5a and 5b in Fig.~\ref{fig-surfs} below), or where whole curves from $J_0$ are mapped to a single point (cases 7a, 7b and 8a in Fig.~\ref{fig-surfs}). It is also violated when $J_0$ does not consist of one or more curves.

We will be especially interested in the existence of `cusps' on otherwise smooth curves of $J_1$.
The term `cusp' is used in different contexts in dynamical systems.
First, the sense in which we use `cusp' in this paper, is to indicate a certain non-smooth point along a plane curve where the curve to either side of the cusp has a common tangent, but as one passes through the cusp point, the tangent vector switches direction; second, it is used in singularity theory for a distinguished point on of a map of the plane which is locally conjugate to 
$(x,y) \mapsto (x, xy-y^3)$ near the origin \cite{Arnold, GG}; and third, it is used in bifurcation theory to denote a certain codimension-two bifurcation point along an otherwise smooth saddle-node bifurcation curve.
Since our focus is on the critical curves $J_0$ and their images $J_1$ as plane curves, in this paper `cusp' is used in the first sense.
We note, however, all the plane curve cusps we consider in this paper are on $J_1$, and in each case, the corresponding map of the plane appears to satisfy the additional conditions to be a cusp in the singularity sense for a map of the plane.
Such plane maps have three local preimages `inside' the cusped $J_1$ curve, and a single local preimage `outside' the cusped curve.
Consistent with this singularity theory model in the second case above, all the cusps we verify in this paper are `of order $3/2$', as in curves parametrized by
$t\mapsto (at^2, bt^3)$ for nonzero constants $a$ and $b$;
cusps with this parametrization are tangent to the positive $x$-axis when $a>0$.

Since the partial derivatives of a quadratic are linear functions of $x$ and $y$, the determinant of the two-by-two Jacobian derivative matrix $DF(x,y)$ is a quadratic in $x$ and $y$:

\begin{align}
\det &\left( DF(x,y)\right)=\begin{vmatrix}
2a_0x+a_1y+a_3 &a_1x+2a_2y+a_4\\
\nonumber
 2b_0x+b_1y+b_3 &b_1x+ba_2y+b_4
\end{vmatrix}\\
\nonumber
&=
(2a_0x+a_1y+a_3)(b_1x+ba_2y+b_4)-
(a_1x+2a_2y+a_4)(2b_0x+b_1y+b_3)\\
\nonumber
&\equiv Ax^2+Bxy+Cy^2+Dx+Ey+F\\
&=2X_{01}x^2+4X_{02}xy+2X_{12}y^2+(2X_{04}-X_{13})x+(X_{14}-2X_{23})y + X_{34}
\label{eq-jacdet}
\end{align}

\noindent
where
$X_{ij}=a_ib_j-a_jb_i=\begin{vmatrix}a_i&a_j\\b_i&b_j\end{vmatrix}$.
The $X_{ij}$ notation is taken from \cite{Nienthesis}.
\noindent
Thus $J_0$ is a conic section, possibly degenerate.

For use later in the paper, we recall the following standard results about conic sections.
Consider the general conic section, which is the zero set of 

\[
\begin{pmatrix}
x&y&1
\end{pmatrix}
\begin{pmatrix}
A& B/2& D/2\\
B/2 &C &E/2\\
D/2  &E/2&F                  
\end{pmatrix}
\begin{pmatrix}
x\\y\\1
\end{pmatrix}
= Ax^2 + Bxy +Cy^2+Dx+Ey+F.
\]

\noindent
Let $\mathcal{D}=B^2-4AC$, and 

\begin{equation}
\label{eq-Delta}
\Delta \equiv \begin{vmatrix}
A& B/2& D/2\\
B/2 &C &E/2\\
D/2  &E/2&F
\end{vmatrix}
\end{equation}

\noindent
The conic section is considered {\it nondegenerate} if $\Delta \ne 0$.  A nondegenerate conic section
is an ellipse if $\mathcal{D} <0$, a hyperbola if 
$\mathcal{D} >0$, and a parabola if $\mathcal{D} =0$. 
In the case of an ellipse, which requires $A$ and $C$ to be nonzero and have the same sign, it is real if $C\Delta<0$, a point if $\Delta=0$, and imaginary if $C\Delta>0$.  
(The case of an imaginary ellipse is not possible for conic sections that arise as Jacobians of quadratic maps; see section \ref{sec:discussion}.) 
A degenerate hyperbola is a pair of intersecting lines. 
A degenerate parabola yields two parallel lines, possibly coinciding, and possibly imaginary.
(As for the imaginary ellipse, the imaginary pair of lines is not possible as the singular set of a quadratic map.)
In addition, the Jacobian determinant of a quadratic map can fail to have any nonzero quadratic terms ($A=B=C=0$). 
In this case, the zero set is a line if at least one of $D$ or $E$ is nonzero, the empty set if, in addition $D=E=0$ but $F \ne 0$, and the whole plane if all coefficients, including $F$ vanish.
These conic section facts determine the classification of $J_0$.

{\bf Results.}
It turns out that in many cases, the geometry of $J_1$ is
completely determined by $J_0$.
In other cases, there are several possibilities for $J_1$.
The following theorem provides a complete list of all
possibilities for both $J_0$ and $J_1$.
It therefore provides a complete classification of critical sets and their images for all quadratic maps.

\begin{theorem}
\label{th-J0J1}
The $J_0$-$J_1$ classification theorem.
Let $F:\mR^2 \rightarrow \mR^2$ be a quadratic map of $\mR^2$ with at least one nonzero quadratic term.
Then $J_0$ and $J_1$ take on one of the following forms:

\begin{enumerate}
\item $J_0$ is empty; $J_1$ is empty
\item $J_0$ is a point; $J_1$ is a point
\item $J_0$ is an ellipse;
$J_1$ is a closed curve with three cusp points
\item $J_0$ is a hyperbola; $J_1$ consists of two curves, one smooth, and the other smooth except for a single cusp; each curve is the image of one branch of the hyperbola
\item $J_0$ is a pair of intersecting lines; $J_1$ is one of the following:
\begin{enumerate}
\item the union of two rays emanating from the same point
\item the union of a ray and a parabola sharing a common point
\end{enumerate}
\item $J_0$ is a parabola; $J_1$ a curve with a single cusp

\item $J_0$ is a pair of parallel lines
\begin{enumerate}
\item distinct lines: $J_1$ is the union of a line and a point. One of the lines in $J_0$ maps to the line in $J_1$ and the other line maps to the point in $J_1$.
\item coincident lines: $J_1$ is a point.
\end{enumerate}
\item $J_0$ is a simple line; $J_1$ is one of the following.  
\begin{enumerate}
\item a point
\item a line
\item a parabola
\end{enumerate}
\item $J_0$ is all of $\mR^2$; $J_1$ is one of the following:
\begin{enumerate}
\item a line
\item a ray
\item a parabola
\end{enumerate}
\end{enumerate}
\end{theorem}

Examples: The existence of all the cases in Theorem \ref{th-J0J1} is provided by the following examples for $F(x,y)$, illustrated in Fig.~\ref{fig-surfs}.
\begin{enumerate}
\item  $(1-ax^2+y,bx)$, $b \ne 0$.
$\det(DF(x,y))=-b$.
These are the invertible Henon maps as long as $b \ne 0$. 
$J_0$ and $J_1$ are empty.
Every point has a unique preimage.
\item $(x^2-y^2+c_1, 2xy+c_2)$ or $(x^2-y^2+c_1, xy+c_2)$.
The first example family is equivalent to $z^2+c$ in complex coordinates; the second family is not complex analytic.
In the first family, $\det(DF(x,y))=4(x^2+y^2)$, and in the second family, $\det(DF(x,y))=2(x^2+y^2)$.
In both families, $J_0$ is $(0,0)$ and $J_1$ is $(c_1, c_2)$.
The plane is double-covered except for $(c_1,c_2)$, which has the origin as its only preimage.
Circles centered at the origin map to ellipses (circles in the complex analytic example) centered at $(c_1,c_2)$, with the image ellipses double-covered.
\item $(x^2-y^2+2x, 2xy-2y)$.
This is equivalent to $z^2+2\zbar$ in complex coordinates.
$\det(DF(x,y))=4(x^2+y^2)-4$, so
$J_0$ is the unit circle; $J_1$ is a deltoid (a hypocycloid with 3 cusps).
The exterior of the deltoid has two preimages;  the interior has four preimages; points on $J_1$ have three preimages, except the three cusp points which have two preimages.
\item $(x^2+y^2+2x, 2xy-2y)$. 
$\det(DF(x,y)=4(x^2-y^2)-4$, so $J_0$ is the hyperbola 
$x^2-y^2=1$; the image of the left branch is smooth; the image of the right branch has a single cusp at $(3,0)$, the image of $(1,0)$.
The number of preimages changes from zero to the left of the smooth branch of $J_1$, two in between the two branches of $J_1$, and four to the right of the cusped branch.
The smooth piece of $J_1$ has only the left branch of the $J_0$ hyperbola as a preimage.
The right (cusped) branch of $J_1$ has unique preimages points on the right branch of the $J_0$ hyperbola, and two additional preimages, except for the cusp point at $(3,0)$ which has only one additional preimage besides $(1,0)$.
\item $J_0$ is a pair of intersecting lines:

\begin{enumerate}

\item $(x^2+y, y^2)$. 
$\det(DF(x,y))=4xy$, so $J_0$ is the union of the two axes.  $J_1$ is the union of a ray -- the nonnegative $x$ axis -- and the parabola $\{ y=x^2\}$, both emanating from the origin.
The part of the second quadrant to the left of the $J_1$ parabola and the third and fourth quadrants have no preimages.
The region `inside' the parabola has two preimages.
The region in the first quadrant but not inside the parabola has four preimages.
The left branch of the $J_1$ parabola has unique preimages (on the negative $y$ axis), the right branch of the $J_1$ parabola has two additional preimages (three total) besides the positive $y$ axis.
The positive $x$ axis has two preimages (one each on the positive and negative $x$ axis), and the origin has only itself as a preimage.

\item $(x^2, y^2)$. 
$\det(DF(x,y))=4xy$, so, again, $J_0$ is the union of the two axes. $J_1$ is the union of two rays -- the nonnegative axes -- emanating from the origin.
There are no preimages from the second, third, and fourth quadrants.
The first quadrant has four preimages.  The positive axes have two preimages, and the origin has only itself as a preimage.

\end{enumerate}
\item $(\frac{1}{2}x^2+y, xy)$. 
$\det(DF(x,y))=x^2 -y$, and $J_0$ is the parabola $y=x^2$; $J_1$ a curve with a single cusp, at $(0,0)$.
Points to the left of the cusped $J_1$ curve have one preimage; points to the right have three preimages; points on $J_1$ have two preimages, one in addition to the preimage on the $J_0$ parabola; the origin has only itself as a preimage.

\item $J_0$ is a pair of parallel lines
\begin{enumerate}
\item $J_0$ is a pair of distinct parallel lines:
$\frac{1}{2}x^2+\epsilon x, xy-\epsilon y)$. 
$\det(DF(x,y))=(x+\epsilon)(x-\epsilon)$, 
$J_0= \{x=-\epsilon\} \bigcup \{x=\epsilon\}$;
$J_1$ is the union of a line and a point; $\{x=-\epsilon\}$ maps to the line $\{x=-\epsilon^2/2\}$; $\{x=\epsilon\}$ maps to the point $(3\epsilon^2/2,0)$.
Fig \ref{fig-surfs} displays this case for $\epsilon=1$.

\item $J_0$ is a pair of coincident lines: $(\frac{1}{2}x^2,xy)$.
$\det(DF(x,y))=x^2$, so $J_0$ is the line $\{x=0\}$; it is distinguished from the next subcase because the line for $J_0$ is really a coincident pair of  lines; it is a limit of case 7a, as $\epsilon$ approaches zero.  $J_1$ is the origin; the $y$ axis has no preimages except for the origin which has the whole $y$ axis as preimages; all points in the right-half-plane have two preimages; points in the left-half-plane have no preimages.

\end{enumerate}

\item $J_0$ is a simple line. 

\begin{enumerate}

\item $J_1$ is a point: $(x,xy)$.
$\det(DF(x,y))=x$, and $J_0$ is the line $\{x=0\}$; $J_1$ is the origin; the $y$ axis has no preimages except for the origin which has the whole $y$ axis as preimages; all other points have unique preimages.

\item  $J_1$ is a line: $(x^2,y)$. $\det(DF(x,y))=2x$, and both $J_0$ and $J_1$ are the $y$ axis.  The right-half-plane has two preimages (all points on the $y$-axis are fixed); the $y$ axis has unique preimages; the left-half-plane has no preimages.

\item $J_1$ is a parabola: $(x^2+y^2,y)$. $\det(DF(x,y))=2x$, and $J_0$ is the $y$ axis; $J_1$ is the parabola $\{x=y^2\}$.
The right of the $J_1$ parabola has two preimanges;
the parabola has unique preimages; 
the left of the parabola has no preimages.
\end{enumerate}
\item Examples with $J_0 = \mR^2$. In all cases,
$\det(DF(x,y))=0$. These cases are not included in Fig.~\ref{fig-surfs} since they are easily visualized by looking at their formulas.
\begin{enumerate}
\item $(x^2-y^2,0)$; $J_1$ is a line: the $x$ axis.
\item $(x^2,0)$; $J_1$ is a ray: the nonnegative $x$ axis.
\item $(x^2,x)$; $J_1$ is the parabola $\{x=y^2\}$.
\end{enumerate}
\end{enumerate}


\begin{figure}
\label{fig-surfs}
\begin{tabular}{|c|c|c|}

\hline
\includegraphics[width=2in]{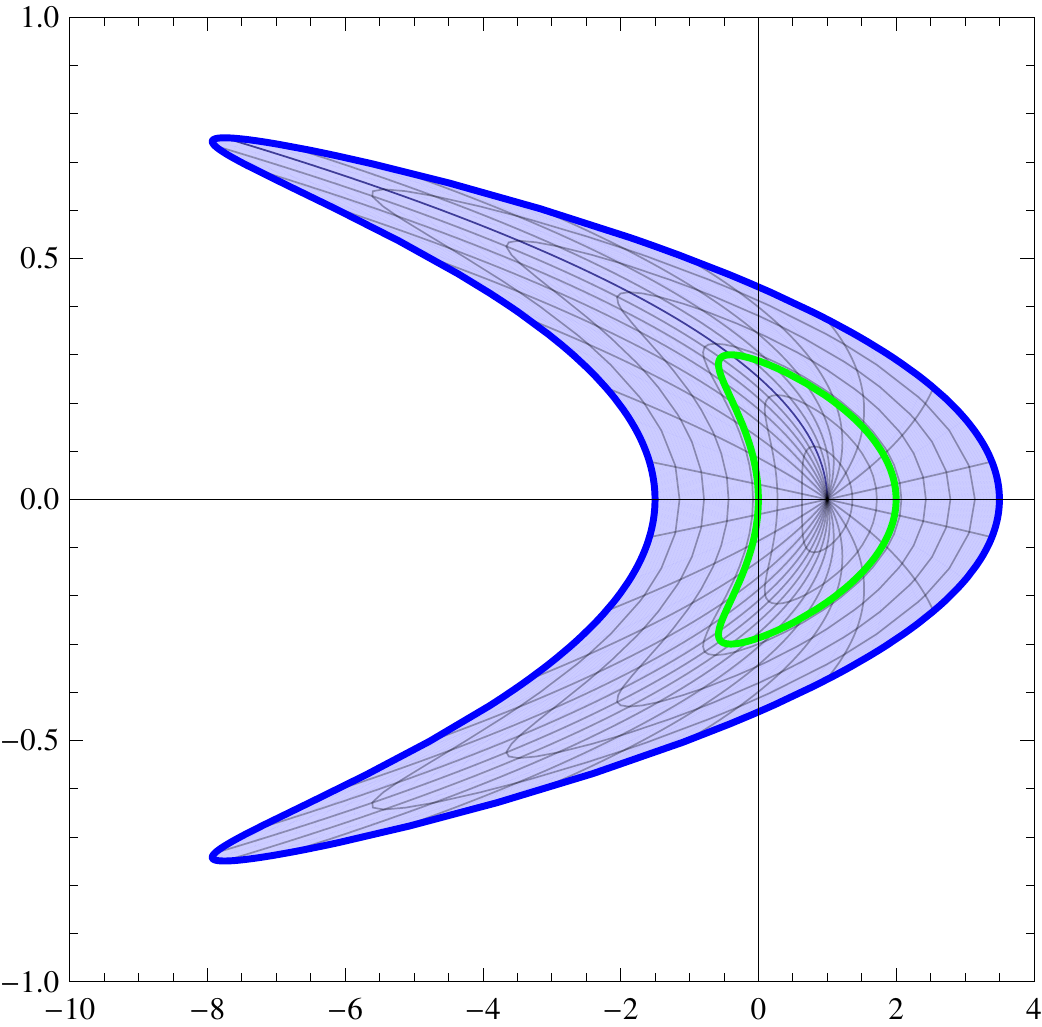}
&\includegraphics[width=2in]{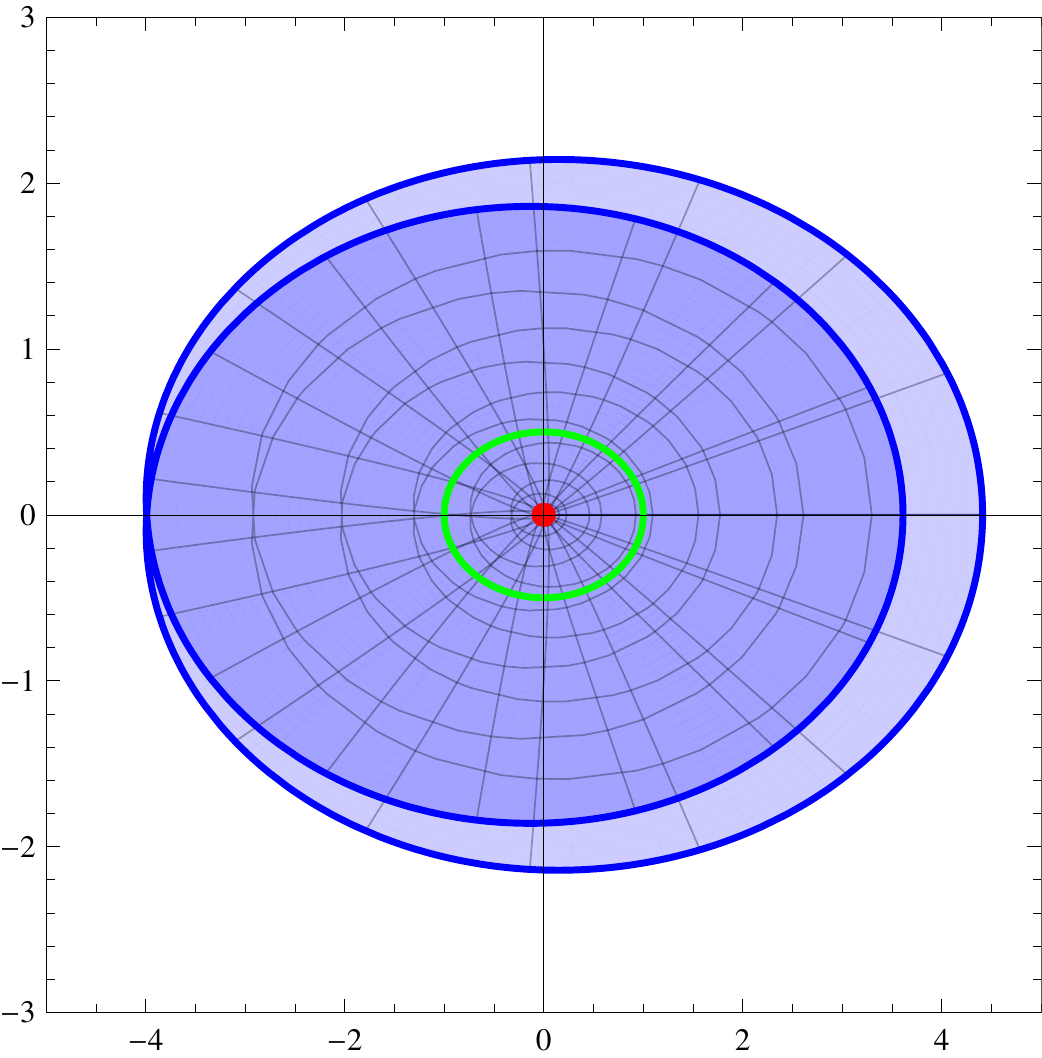}
&\includegraphics[width=1.8in]{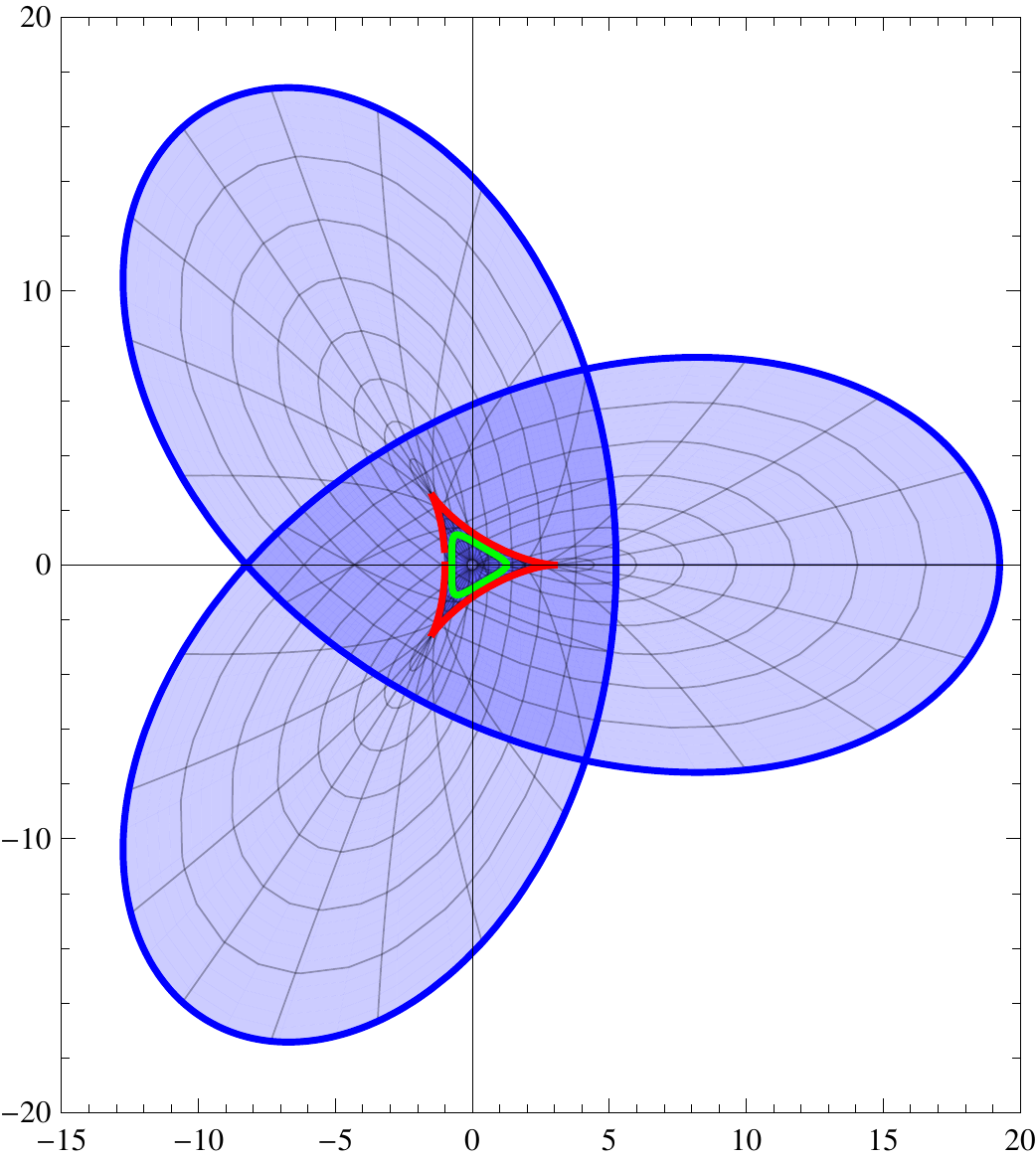}\\
1. $(1-1.4x^2+y, .3x)$&2. $(x^2-y^2, xy)$&3. $(x^2-y^2+2x, 2xy-2y)$\\

\hline
\includegraphics[width=1.8in]{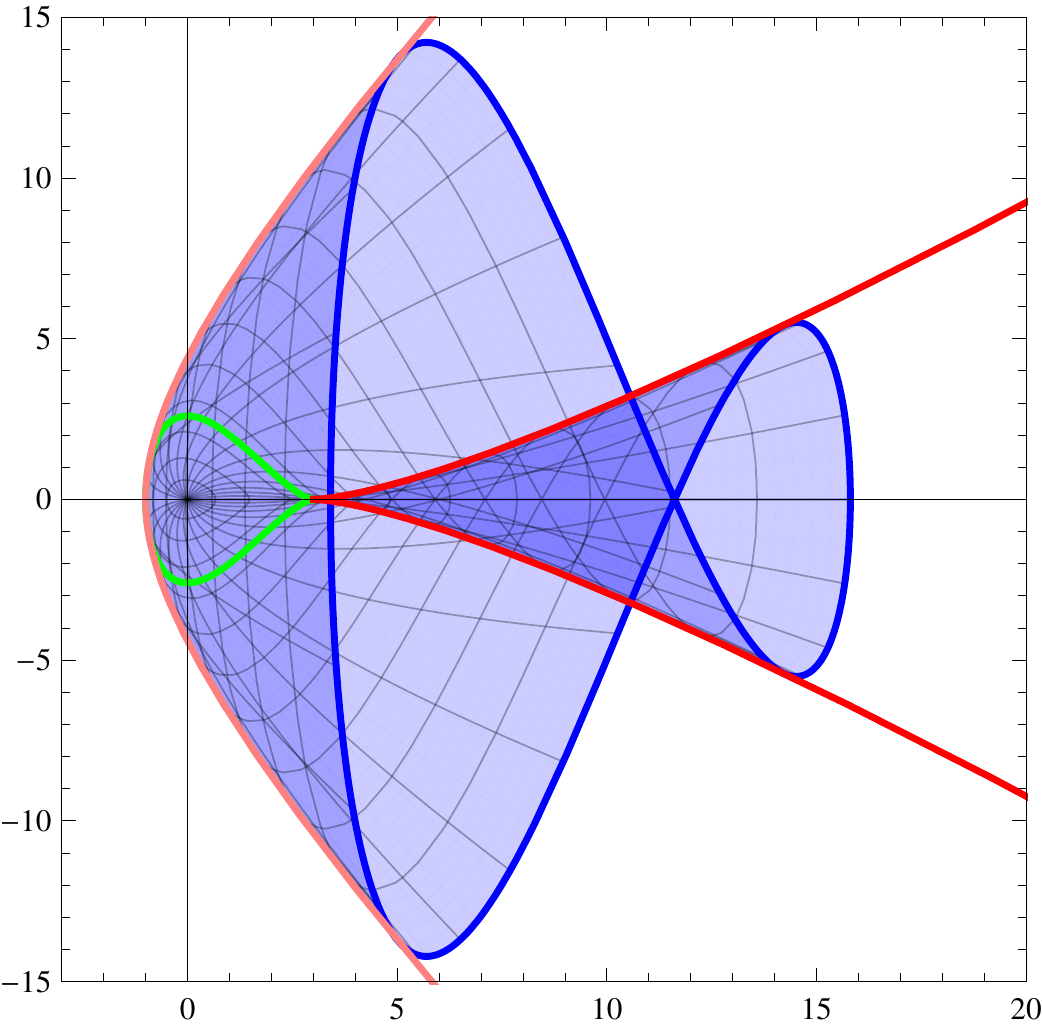}
&\includegraphics[width=1.8in]{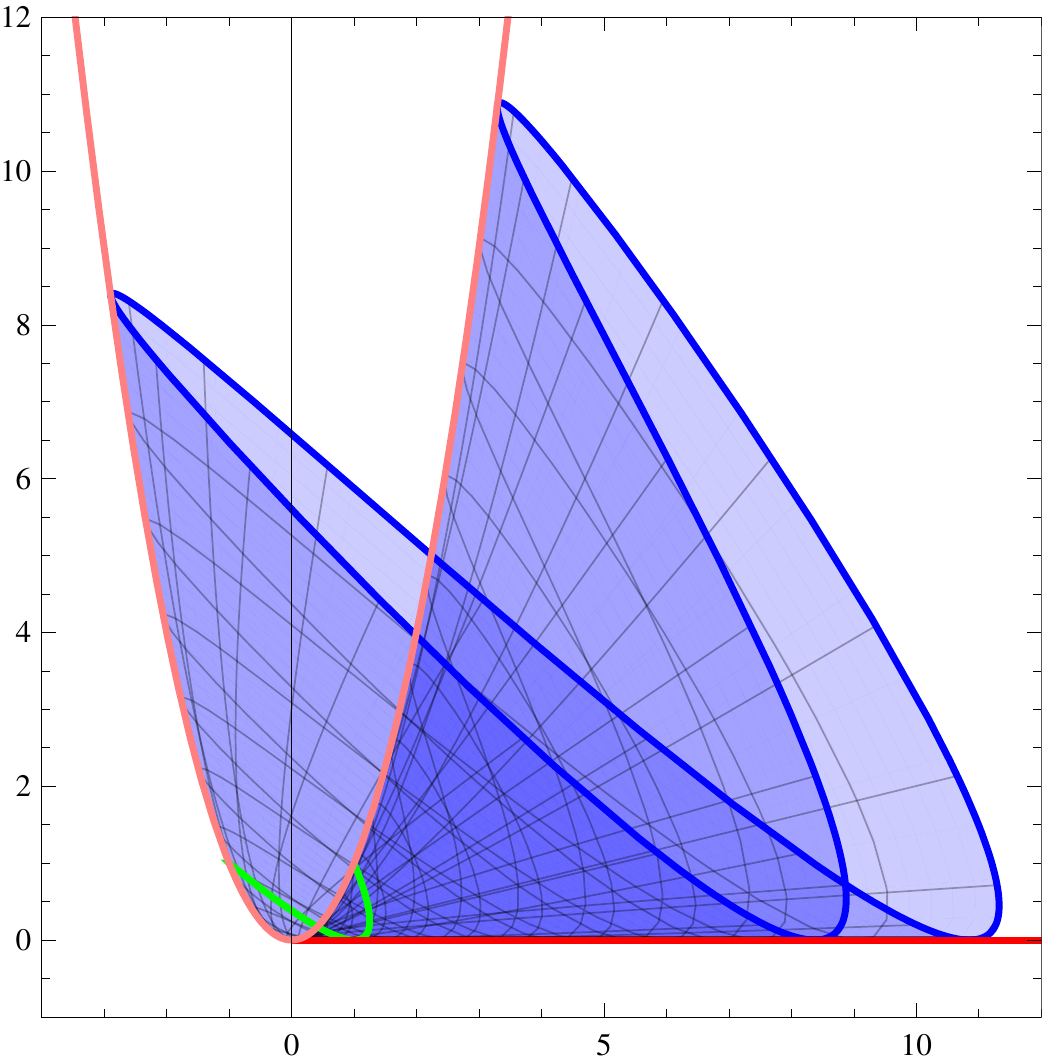}
&\includegraphics[width=1.9in]{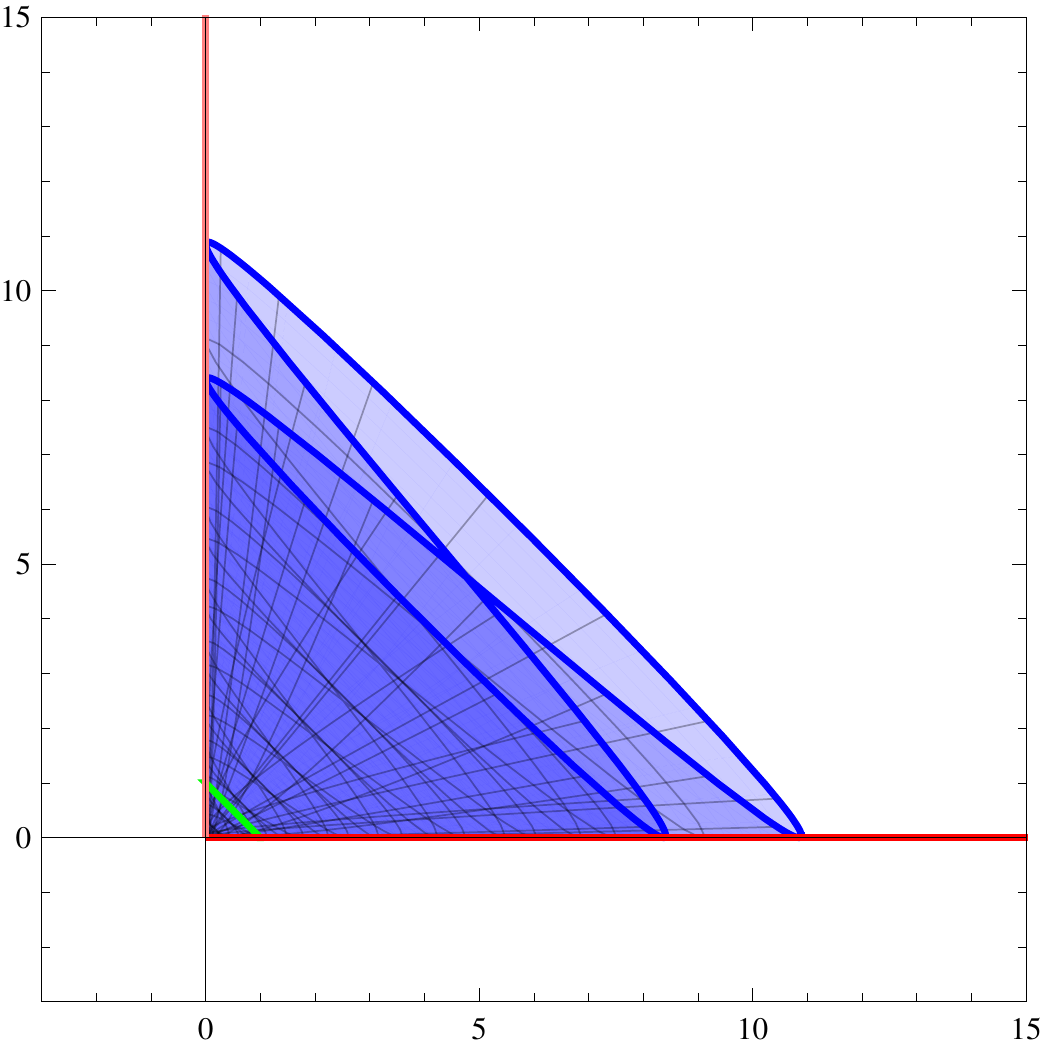}\\
4. $(x^2+y^2+2x, 2xy-2y)$&5a. $(x^2+y, y^2)$&5b. $(x^2, y^2)$\\
\hline
\includegraphics[width=1.6in]{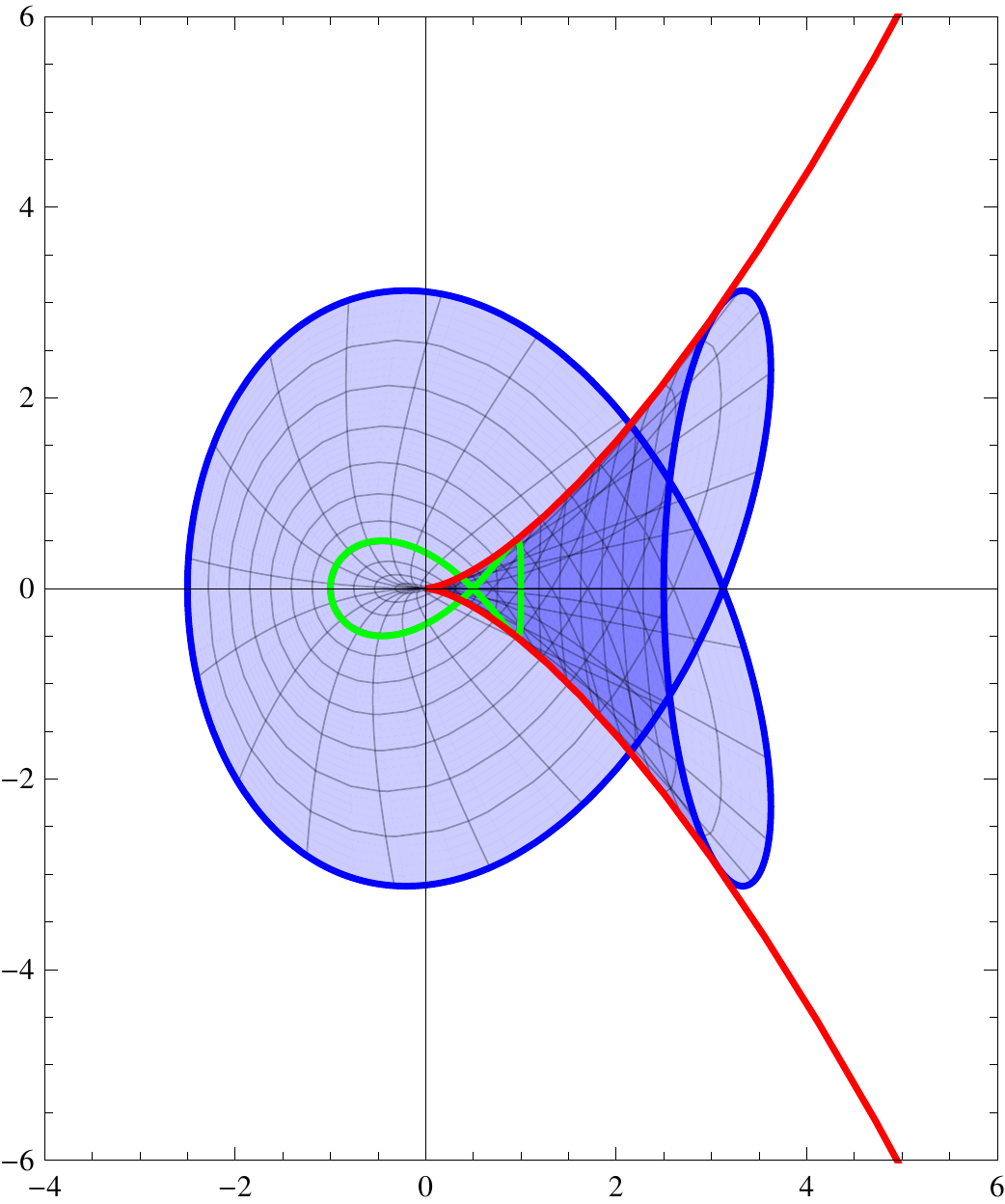}
&\includegraphics[width=1.4in]{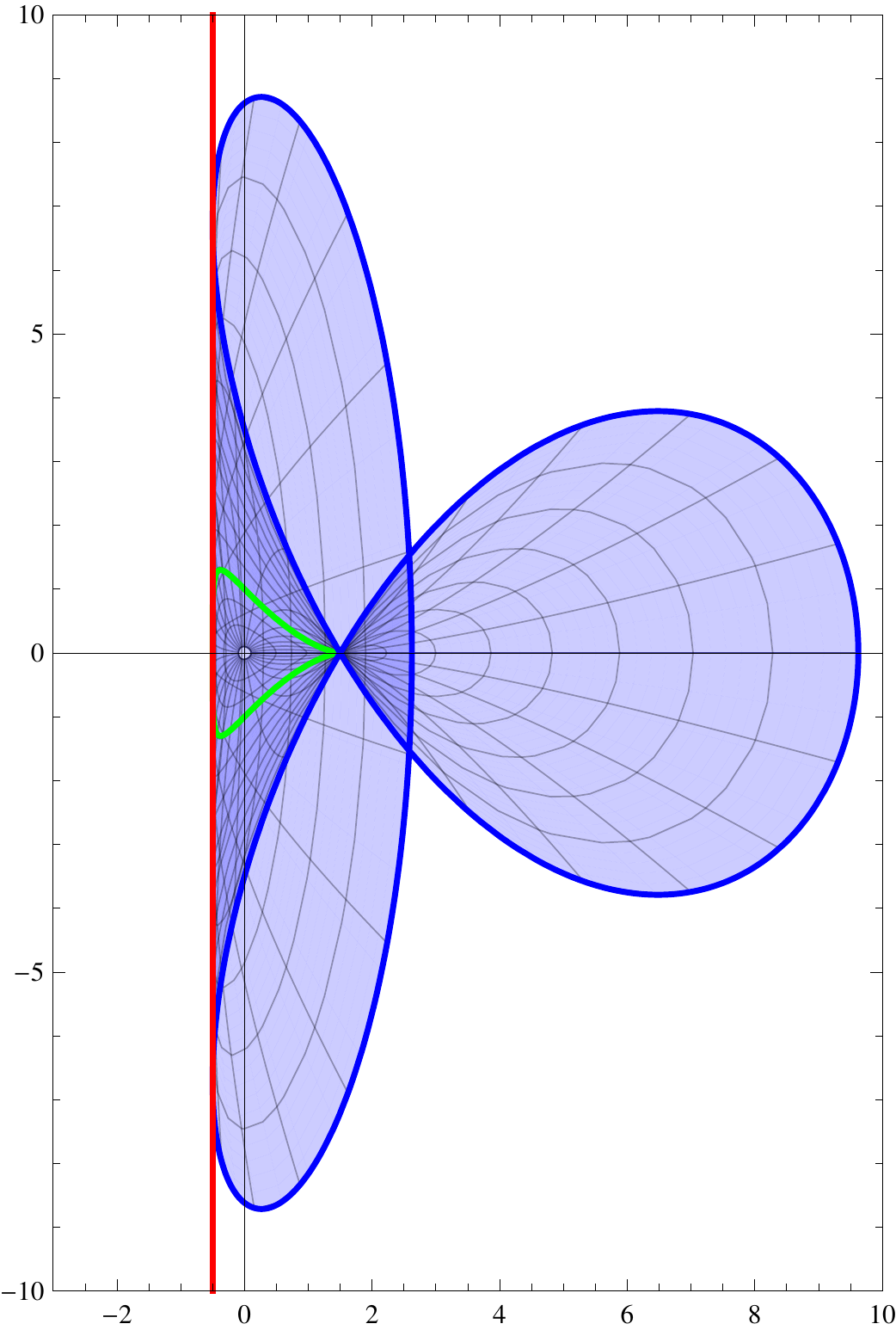}
&\includegraphics[width=1.9in]{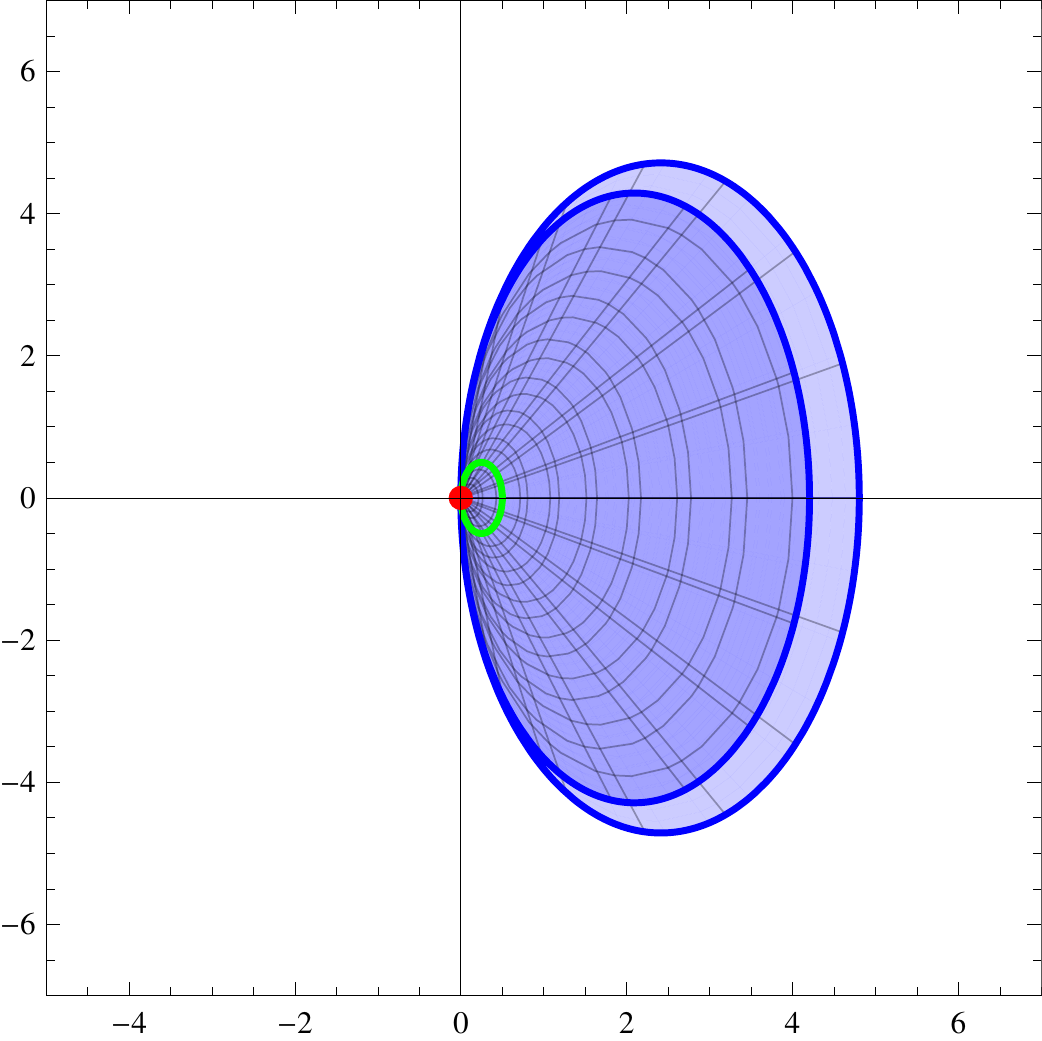}\\
6. $(\frac{1}{2}x^2+y, xy)$&7a. $(\frac{1}{2}x^2+x, xy-y)$&7b. ($\frac{1}{2}x^2,xy)$\\
\hline
\includegraphics[width=1.5in]{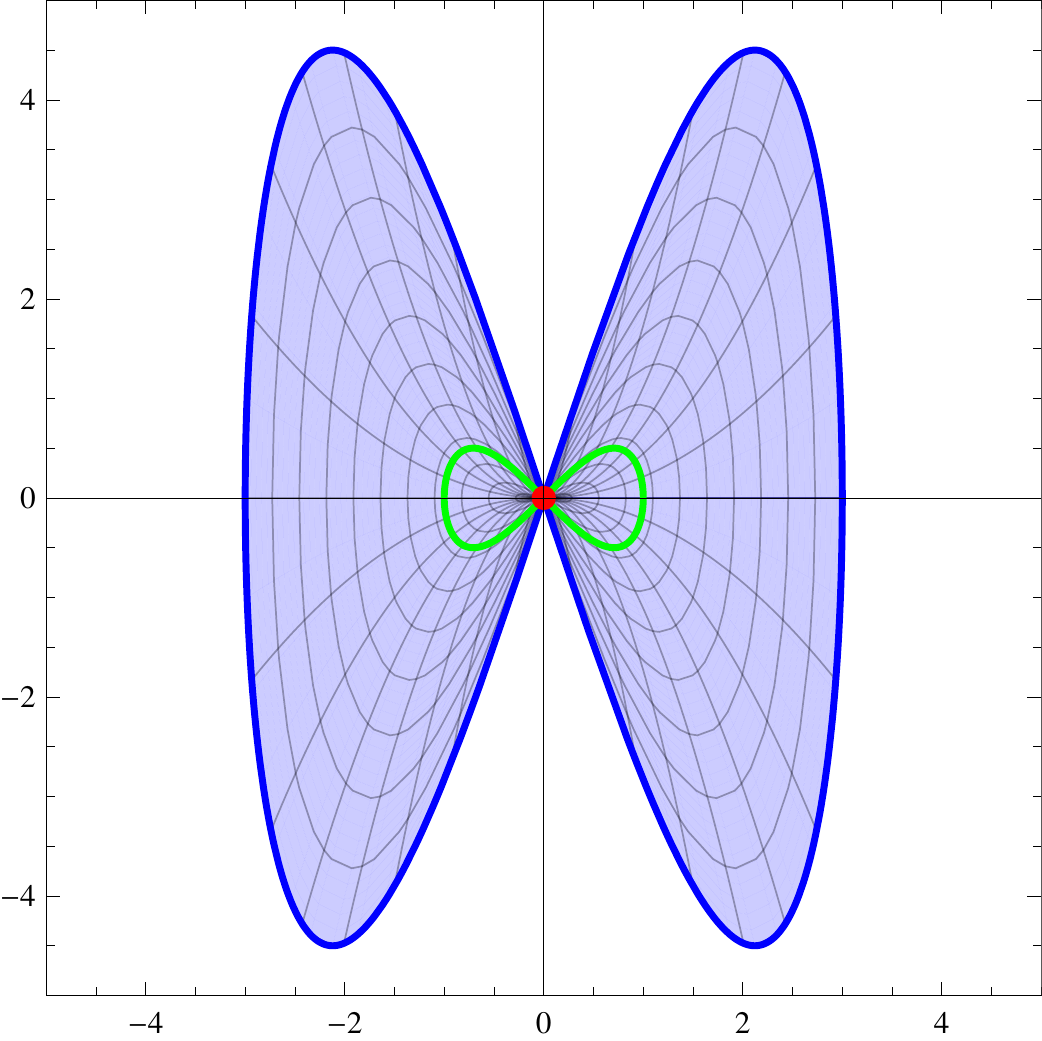}
&\includegraphics[width=2in]{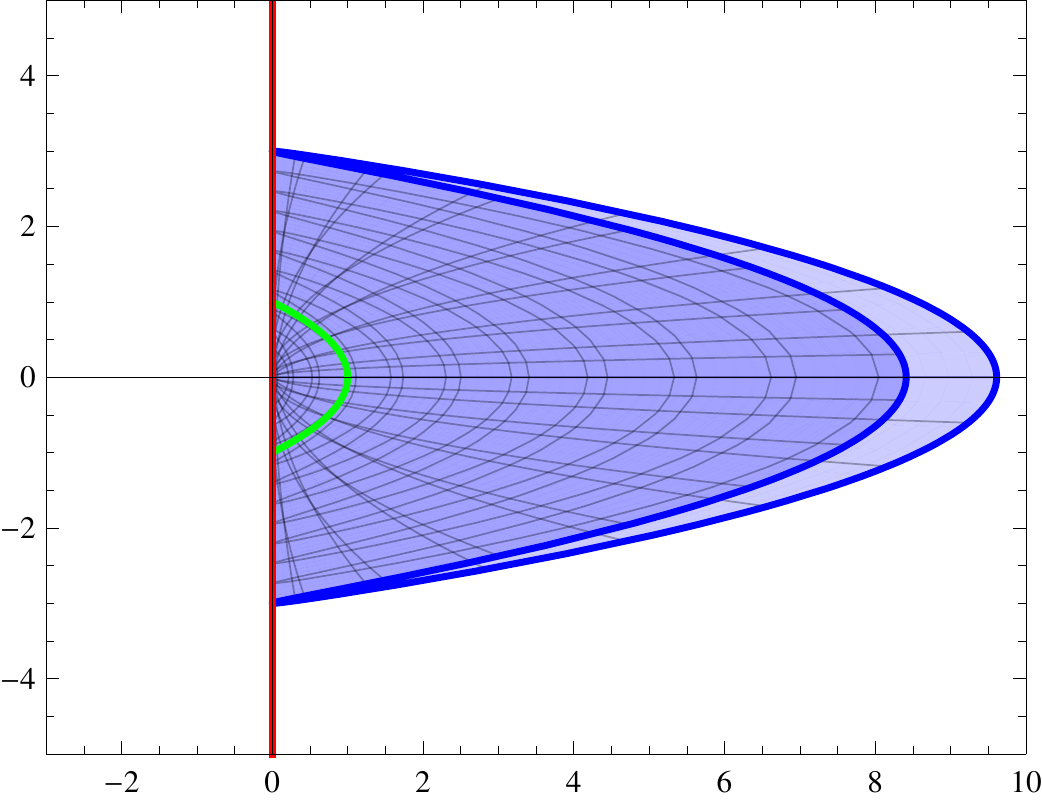}
&\includegraphics[width=1.9in]{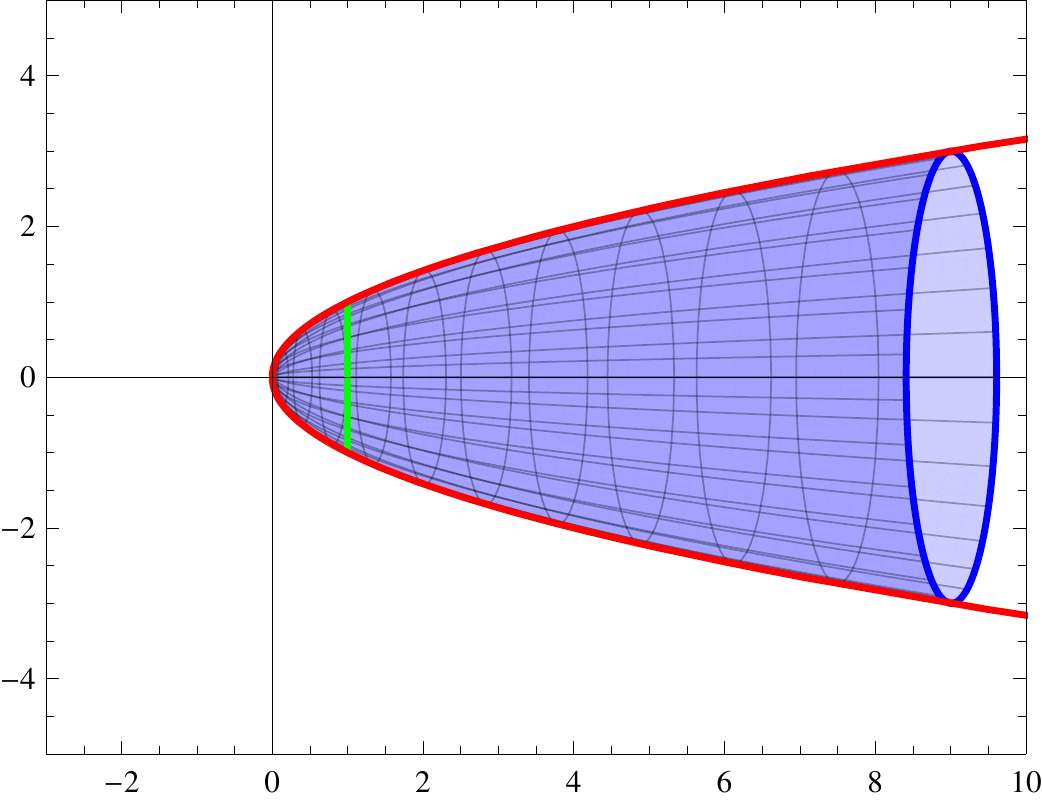}\\
8a. $(x,xy)$&8b. $(x^2,y)$&8c. $(x^2+y^2,y)$\\
\hline
\end{tabular}
\caption{Images of disks illustrating Theorem~\ref{th-J0J1}.
Disks have radius r, and center $(x_0, 0)$.  Nonzero centers were chosen to prevent plotting curves with coincident projection.
$(r,x_0)=$  1. $(2.5, 0)$, 2. $(2, 0.1)$, 3. $(3.5, 0)$, 4. $(3.1, 0)$, 5a.  $(3.1, 0.2)$, 5b. $(3.1, 0.2)$, 6. $(2.5, 0)$, 7a. $(3.5, 0)$, 7b.  $(3, 0.1)$, 8a. $(3, 0)$, 8b. $(3, 0.1)$, 8c. $(3, 0.1)$.
The image of the unit circle is in green; $J_1$ is in red. }
\end{figure}

\section{Proofs}
\label{s:proofs}

Before we begin the proof of Theorem \ref{th-J0J1}, we establish some useful (standard) lemmas and identities.

\begin{lemma}
\label{lemma:equivalence}
Let $h$ and $k$ be two diffeomorphisms of $\mR^2$, and $F:\mR^2 \rightarrow \mR^2$ a smooth two-dimensional map.  Define the map $G:\mR^2 \rightarrow \mR^2$ by $G=k \circ F \circ h^{-1}$.
Then $J_0^G=h(J_0^F)$ and $J_1^G=k(J_1^F)$.
$F$ and $G$ are said to be {\sc map equivalent}; map equivalent functions are called {\sc geometrically equivalent} in \cite{DGRRV}, or just {\sc equivalent} in singularity theory \cite{GG}.
\end{lemma}
\begin{proof}
Differentiate $k\circ F(x,y) = G \circ h(x,y)$ using the chain rule: 
\[Dk(F(x,y)) DF(x,y)=DG(h(x,y) Dh(x,y).\]
Use the fact that $h$ and $k$ are diffeomorphisms to get $\det(DF(x,y))=0$ iff $\det(DG(h(x,y))=0$.
That is, $h(J^F_0) = J_0^G$.
For the image result, $J_1^G =G(J_0^G)=G(h(J^F_0))=k(F(J_0^F))=k(J_1^F)$.
\end{proof}

\begin{corollary}
\label{cor:stdconic}
When $h$ and $k$ are both affine diffeomorphisms of $\mR^2$ (so $F$ and $G$ are {\sc affinely map equivalent}),
and $J_0^F$ (resp. $J_1^F$) is one of the following: ellipse, hyperbola, parabola, line, ray, point, then
$J_0^G$ (resp. $J_1^G$) has the same geometric description.
\end{corollary}
\begin{proof}
Nonsingular affine transformations of the plane preserve all of the geometric objects listed in the Corollary.
\end{proof}

Since it is straightforward to construct an affine diffeomorphism taking any conic section to one in a standard form, this corollary allows us to assume a convenient form for $J_0$ in the proof of Theorem \ref{th-J0J1} below.

\begin{lemma}
\label{lemma:identities}
Consider the cross product determinants 
$X_{ij}=a_ib_j-a_jb_i$ following equation (\ref{eq-jacdet}).  Then the following identities are easily verified for any set of $i,j,k,l$:
\begin{align}
X_{ij}&=-X_{ji}\tag{X.ij}\\
a_k X_{ij}+a_i X_{jk}+a_j X_{ki} &=0\tag{aX.kij}\\
b_k X_{ij}+b_i X_{jk}+b_j X_{ki} &=0\tag{bX.kij}\\
X_{ij}X_{kl} - X_{ik}X_{jl} + X_{il}X_{jk} &=0\tag{XX.ijkl}.
\end{align}

\end{lemma}

\begin{proof}

All equations are easily verified once it is noticed that they can be written in terms of determinants.
Eq.~(X.ij) can be written as
$\begin{vmatrix}
a_i &b_i\\ a_j &b_j
\end{vmatrix}
=-\begin{vmatrix}
a_j &b_j\\ a_i &b_i
\end{vmatrix}$.
Eq.~(aX.kij) can be expressed as 
$\begin{vmatrix}
a_k &a_k &b_k\\ a_i &a_i &b_i \\ a_j &a_j &b_j\\
\end{vmatrix}
=0$.
Eq.~(bX.kij) can be similarly verified.
When the left-hand-side of eq.~(XX.ijkl) is multiplied out in terms of $a_i$'s and $b_j$'s, the twelve terms each appear twice in the twenty-four terms in the following determinant, which is clearly zero:
$\begin{vmatrix}
a_i &b_i &a_i &b_i\\ a_j &b_j &a_j &b_j \\
a_k &b_k &a_k &b_k\\ a_l &b_l &a_l &b_l
\end{vmatrix}
=0$.
\end{proof}
\noindent
We will refer to these equations for various combinations of indices in the proofs that follow below.

\begin{lemma}
\label{lemma:quadimage}
Consider the plane curve $C$ parametrized by $(x(t), y(t))=(\alpha t^2+\beta t, \gamma t^2+\delta t)$ where $t \in \mR$.
Then $C$ is a 
\begin{enumerate}
\item point if $\alpha=\beta=\gamma=\delta=0$
\item line if 
$\alpha=\gamma=0$ but at least one of $\beta$ and $\delta$ is nonzero.
\item a ray if $\Gamma \equiv \alpha \delta - \beta \gamma=0$ but at least one of $\alpha$ and $\gamma$ is nonzero.
\item a nondegenerate parabola if $\Gamma  \ne 0$.
\end{enumerate}
\end{lemma}

\begin{proof}
\begin{enumerate}
\item The image is the origin for all $t$.
\item $C$ is parametrized by $t(\beta, \delta)$.
\item Assume without loss of generality $\alpha \ne 0$. Then $(\alpha t^2+\beta t, \gamma t^2+\delta t)=(\alpha t^2+\beta t)(1, \gamma/\alpha)$, which is clearly a ray with endpoint corresponding to the value of $t$ which makes $\alpha t^2+\beta t$ a minimum.
Except for the endpoint of the ray, each point on the ray has two preimage values of $t$.
\item If $\alpha=0$ then $\gamma \ne 0$ and $\beta \ne 0$ and $t$ can be eliminated to give $y=\frac{\gamma}{\beta^2}x^2 + \frac{\delta}{\beta}x$, which is clearly a nondegenerate parabola.  Similarly $\gamma=0$ leads to a nondegenerate parabola, but with axis parallel to the $x$ axis.
If both $\alpha$ and $\gamma$ are nonzero, compute $\gamma x(t)-\alpha y(t)$ to eliminate the $t^2$ terms and solve for $t$ to get $t=\frac{\alpha y - \gamma x}{\Gamma}$.  
Substitute this into the formula for $x(t)$ to get a quadratic in $x$ and $y$.  Computing the coefficients of the quadratic terms and showing `$B^2-4AC=0$' since $B^2=4AC=4\frac{\alpha^4\gamma^2}{\Gamma^2}$ (recall eq.~(\ref{eq-jacdet})) verifies that $(x(t),y(t))$ is a parabola.
The quantity $\Delta$ can be computed to be $-\frac{\alpha^3}{4\Gamma^2} \ne 0$ (recall eq.~(\ref{eq-Delta})), which shows that the parabola is nondegenerate.
\end{enumerate}
\end{proof}

\begin{lemma}
\label{lemma:cusp}
Consider the smooth plane curve $C$ parametrized by
$\boldsymbol\alpha(t) =$
$ (\alpha_1(t), \alpha_2(t))$.
If $\boldsymbol\alpha'(t_0)=(0,0)$, and $\gamma(t_0) \equiv \alpha_1''(t_0)\alpha_2'''(t_0) - \alpha_1'''(t_0)\alpha_2''(t_0) \ne 0$, then $C$ has a cusp at $\boldsymbol\alpha(t_0)$ with order of tangency equal to $3/2$.

\end{lemma}

\begin{proof}
Translate $\boldsymbol\alpha(t_0)$ to the origin in the plane and expand around $t=t_0$. After replacing $t-t_0$ with $t$ and $\boldsymbol\alpha(t)-\boldsymbol\alpha(t_0)$ with $\boldsymbol\alpha(t)$, we get
$\boldsymbol\alpha(t)=(\frac{\alpha_1''(t_0)}{2!}t^2 + \frac{\alpha_1'''(t_0)t^3}{3!} + O(t^4),  \frac{\alpha_2''(t_0)t^2}{2!}+\frac{\alpha_2'''(t_0)}{3!}t^3 + O(t^4))$.
\begin{itemize}
\item Case 1: $\alpha_2''(t_0)=0$. Then the lowest order terms in $t$ of $(\alpha_1(t), \alpha_2(t))$: $(\frac{\alpha_1(t_0)}{2!} t^2, \frac{\alpha_2(t_0)}{3!} t^3)$ give the parametric version of the standard cusp of order $3/2$ \cite{Arnold, GG}.
Note that $\gamma(t_0) \ne 0$ implies that both $\alpha_1''(t_0)$ and $\alpha_2'''(t_0)$ are nonzero.
\item Case 2: $\alpha_2''(t_0) \ne 0$.  
We can put the curve into the form of Case 1 by multiplying $\boldsymbol\alpha(t)$ by a (nonsingular) matrix:

$\left[
\begin{smallmatrix}
\alpha_1''(t_0) & \alpha_2''(t_0) \\
-\alpha_2''(t_0) & \alpha_1''(t_0)
\end{smallmatrix}
\right]
\left[
\begin{smallmatrix}
\alpha_1(t)\\
\alpha_2(t)
\end{smallmatrix}
\right] =
\left[
\begin{smallmatrix}
(\frac{(\alpha_1''(t_0))^2}{2!} + \frac{(\alpha_2''(t_0))^2)}{2!} t^2 + O(t^3)\\ \frac{\gamma(t_0)}{3!} t^3 + O(t^4))
\end{smallmatrix}
\right]$.

The requirements that $\alpha_2''(t_0) \ne 0$ and $\gamma(t_0) \ne 0$ guarantee that the leading coefficient in each of the two components is nonzero.
Note that multiplication by the matrix is a rescaling
by the determinant of the matrix composed with a rotation by the negative of the angle determined by
the vector of the quadratic coefficients: $(\alpha_1''(t_0), \alpha_2''(t_0))$.
The sides of the cusp are tangent to this vector.
After multiplication, the sides of the cusp are tangent to the positive $x$ axis, as in Case 1.
\end{itemize}

\end{proof}

\subsection{Proof of Theorem \ref{th-J0J1}}

We now proceed to the proof of our main result, Theorem {\ref{th-J0J1}.
The starting point for most cases is to use Cor. \ref{cor:stdconic} to allow us to use a convenient form for each case.
That is, using the notation of Lemma~\ref{lemma:equivalence}, we
replace $F$ with $k\circ F \circ h^{-1}$ where the affine diffeomorphism $h$ is selected to map the singular set for $F$ to a standard form, and $k$ is a rescaling of $x$ and/or $y$ to eliminate any nonzero constant factor of $\det(DF(x,y))$).
For example, in case 3 below, we choose $h$ so that it maps the singular set which is assumed to be an ellipse, to the unit circle $x^2+y^2=1$.
The choice of $k$ allows us to assume that 
$\det(DF(x,y))$ is exactly $x^2+y^2-1$.
This allows us to assign specific values to the six coefficients of $\det(DF(x,y))$ according to eq.~(\ref{eq-jacdet}).
We then use algebraic manipulation, along with our lemmas above, to establish the results.
For completeness, and because our proofs of the two generic cases differ from the proofs in \cite{DGRRV}, we include proofs for the two generic cases in the Appendix. 

\begin{proof}
\begin{enumerate}
\item The image of the empty set is empty.
\item The image of a point is a point.
\item See the Appendix.
\item See the Appendix.

\item $J_0$ is a pair of intersecting lines.
By Corollary \ref{cor:stdconic}, we can assume that $\det(DF(x,y))=xy$.
By equation (\ref{eq-jacdet}),
\begin{align}
2X_{01}&=0\label{eq:ilcoeff0}\\
4X_{02}&=1\label{eq:ilcoeff1}\\
2X_{12}&=0\label{eq:ilcoeff2}\\
2X_{04}+X_{31}&=0\label{eq:ilcoeff3}\\
2X_{32}+X_{14}&=0\label{eq:ilcoeff4}\\
X_{34}&=0\label{eq:ilcoeff5}.
\end{align}

Parametrize $J_0$ by $\{(0,t)\} \bigcup \{(t,0)\}$.
Then $J_1$ is parametrized by
\linebreak
$\{\boldsymbol\alpha(t) \} \bigcup \{(\boldsymbol\beta(t) \}$
where $\boldsymbol\alpha(t) = 
(a_2 t^2 + a_4 t, b_2 t^2+ b_4 t)$
and 
$\boldsymbol\beta(t) = 
(a_0t^2 +a_3 t, b_0 t^2+ b_3 t)$.

Note that $a_2$ and $b_2$ cannot both be zero 
by (\ref{eq:ilcoeff1}).  Similarly, $a_0$ and $b_0$ cannot both be zero.
By Lemma \ref{lemma:quadimage}, if
$X_{24}=0$, then $\boldsymbol\alpha(t)$ is a ray, and if $X_{24} \ne 0$, then $\boldsymbol\alpha(t)$ is a parabola.
Similarly, if
$X_{03}=0$, then $\boldsymbol\beta(t)$ is a ray, and if $X_{03} \ne 0$, then $\boldsymbol\beta(t)$ is a parabola.
We will show that $X_{03}X_{24}=0$, implying that at least one of the branches of $J_1$ is a ray, and the other is either a ray or a parabola.
This follows since (\ref{eq:ilcoeff0}), (\ref{eq:ilcoeff1}), (\ref{eq:ilcoeff2}) and
(aX.201) imply that $a_1=0$.
Similarly, with (bX.201) we get $b_1=0$. 
Therefore, both $X_{14}$ and $X_{13}$ are zero, 
which forces $X_{04}$ and $X_{23}$ to be zero by 
(\ref{eq:ilcoeff3}) and (\ref{eq:ilcoeff4}).
Finally, (XX.0234) implies $X_{03}X_{24}=0$.

\item $J_0$ is a parabola.
By Corollary \ref{cor:stdconic}, we can assume that $\det(DF(x,y))=x^2-y$.
By equation (\ref{eq-jacdet}),
\begin{align}
2X_{01}&=1\label{eq:pcoeff0}\\
4X_{02}&=0\label{eq:pcoeff1}\\
2X_{12}&=0\label{eq:pcoeff2}\\
2X_{04}+X_{31}&=0\label{eq:pcoeff3}\\
2X_{32}+X_{14}&=-1\label{eq:pcoeff4}\\
X_{34}&=0\label{eq:pcoeff5}.
\end{align}

Parametrize $J_0$ by $\{(t, t^2) | t \in \mR \}$.
Then $J_1$ is parametrized by
\linebreak
$\{(\alpha_1(t), \alpha_2(t)) | t \in \mR \}$
where $(\alpha_1(t), \alpha_2(t)) = (a_0 t^2 + a_1 t^3 + a_2 t^4+a_3 t + a_4 t^2, 
b_0 t^2 + b_1 t^3 + b_2 t^4+b_3 t + b_4 t^2)$.
We will find solutions to  so
$(\alpha_1'(t), \alpha_2'(t))=(0,0)$ to find possible cusps.
That is, we solve 
\begin{align}
\alpha_1'(t)&=a_3+2a_0t+2a_4t+3a_1 t^2+4a_2 t^3=0\label{eq:pa1prime}\\
\alpha_2'(t)&=b_3+2b_0t+2b_4t+3b_1 t^2+4b_2 t^3=0\label{eq:pa2prime}
\end{align}

{\it Showing that there is only one possible singular point on $J_1$.}
First we observe that eq.~(aX.201)
with (\ref{eq:pcoeff0}, \ref{eq:pcoeff1}, \ref{eq:pcoeff2}) implies $a_2=0$; similarly using eq.~(bX.201) leads to $b_2=0$.  So $X_{2i}=0$ for all $i$.
Thus, the $t^3$ terms drop out.
Now we eliminate the $t^2$ terms by
$b_1$ (\ref{eq:pa1prime}) $-$  $a_1$ (\ref{eq:pa2prime}),
leaving us with $X_{31}+2X_{01}t+2X_{41}t=0$.
But eqs.~(\ref{eq:pcoeff0}, \ref{eq:pcoeff4}) and
$X_{2i}=0$ leaves us with $X_{31}+3t=0$, or
$t=\frac{X_{13}}{3}$.
Therefore $t=\frac{X_{13}}{3}\equiv T$ is the only possible solution.

{\it Showing $T$ satisfies $(\alpha_1'(T), \alpha_2'(T))=(0,0)$.}
Rearranging eq.~(\ref{eq:pa1prime}) yields 
$\alpha_1'(T)=\alpha_1'(\frac{X_{13}}{3})=
\frac{2}{3}[a_3(\frac{1}{2})+a_0 X_{13}+\frac{a_1}{2} X_{13}^2]+
\frac{2}{3}[a_3+a_4 X_{13}]$. We will show both quantities in square brackets are zero.
The first quantity is zero by (aX.301) after using (XX.0134) with $X_{34}=0$ and $X_{14}=-1$ to replace $X_{30}$ with $\frac{1}{2}X_{13}^2$.  The second quantity is zero by (aX.413).
Similar calculations show that $\alpha_2'(T)$ also equals zero.
Therefore there is exactly one singular point on $J_1$.

{\it Showing that the singular point on $\boldsymbol\alpha(t)$ is a nondegenerate cusp.}
By Lemma \ref{lemma:cusp}, we must show 
$\Gamma(T)\equiv \alpha_1''(T) \alpha_2'''(T) - \alpha_1'''(T) \alpha_2''(T) \ne 0$.
By differentiating eqs.~(\ref{eq:pa1prime}) and (\ref{eq:pa2prime}) this condition becomes

\begin{align}
\Gamma(T)&=
(2(a_0+a_4)+2 a_1 X_{31})(6b_1) - (2(b_0+b_4)+2 b_1 X_{31})(6a_1)\nonumber \\
&=12(X_{01}+X_{41}) = 12(\frac{1}{2} + 1)=18 \ne 0
\end{align}

\noindent
We note that example 6 in Fig.~\ref{fig-surfs}, $(x,y)\mapsto (x^2+y, xy)$, satisfies the stronger condition of being a cusp in the singularity sense as a map of the plane since it is locally {\it map equivalent} (recall Lemma~\ref{lemma:equivalence}) via the diffeomorphisms $h(x,y)=(y,x-y^2)$ ($h$ is the {\it near identity} transformation $(x,y)\mapsto (x-y^2, y)$ composed with $(x,y) \mapsto (y,x)$), and $k$ the identity to the normal form
$(x, xy-y^3)$ for the standard plane map cusp \cite{GG}.
By Golubitsky and Guilleman's analysis of the normal form, this map has a curve of singular points which has a plane curve cusp or order $3/2$ which passes through the origin.

\item $J_0$ is a pair of parallel lines.
\begin{enumerate}
\item If the lines are distinct, by Corollary \ref{cor:stdconic}, we can assume that
\linebreak
$\det(DF(x,y))=x^2-1$.
By equation (\ref{eq-jacdet}),
\begin{align}
2X_{01}&=1\label{eq:plcoeff0}\\
4X_{02}&=0\label{eq:plcoeff1}\\
2X_{12}&=0\label{eq:plcoeff2}\\
2X_{04}+X_{31}&=0\label{eq:plcoeff3}\\
2X_{32}+X_{14}&=0\label{eq:plcoeff4}\\
X_{34}&=-1\label{eq:plcoeff5}.
\end{align}

Parametrize $J_0$ by $\{(1,t)\} \bigcup \{(-1,t)\}$.
Then $J_1$ is parametrized by 
$\{\boldsymbol\alpha(t) \} \bigcup \{\boldsymbol\beta(t) \}$
where $\boldsymbol\alpha(t) = 
(a_0 + a_1 t + a_2 t^2+a_3 + a_4 t, 
b_0 + b_1 t + b_2 t^2+b_3 + b_4 t)$
and 
$\boldsymbol\beta(t) = 
(a_0 - a_1 t + a_2 t^2-a_3 + a_4 t, 
b_0 - b_1 t + b_2 t^2-b_3 + b_4 t)$.

We observe that eq.~(aX.201)
with (\ref{eq:plcoeff0}, \ref{eq:plcoeff1}, \ref{eq:plcoeff1}) implies $a_2=0$; similarly using eq.~(bX.201) leads to $b_2=0$, so 

\begin{align}
\boldsymbol\alpha(t) &= 
(a_0 + a_3 + (a_1+a_4) t, 
b_0 +b_3+ (b_1 +b_4)t)\\
\boldsymbol\beta(t) &= 
(a_0 -a_3 +(a_4- a_1) t, 
b_0 -b_3+(b_4- b_1) t).
\end{align}

We will show that either $a_1+a_4$ and $b_1+b_4$ are both zero and $a_4-a_1$ and $b_4-b_1$ are both nonzero, or 
$a_1+a_4$ and $b_1+b_4$ are both nonzero and $a_4-a_1$ and $b_4-b_1$ are both zero.
This will guarantee that $J_1$ is the union of a line and a point.

First note that $a_2=0$ and $b_2=0$ implies $X_{2i}=0$ for all $i$.
Thus, (\ref{eq:plcoeff4}) implies $X_{14}=0$.
By (\ref{eq:plcoeff3}) and (XX.0134),
$X_{13}^2=1$.  So $X_{13}= \pm 1$. 
When $X_{13}=1$, (aX.413) implies $a_1-a_4=0$
and (bX.413) implies $b_1-b_4=0$.
Note that $a_1+a_4$ and $b_1+b_4$ cannot both be zero without forcing $a_4$ and $b_4$ to both be zero, which would force $X_{34}$ to be zero, contradicting 
(\ref{eq:plcoeff5}).
Thus, $\boldsymbol\alpha(t)$ is a line and $\boldsymbol\beta(t)$ is a point.
Similarly, $X_{13}=-1$ implies
$\boldsymbol\alpha(t)$ is a point and $\boldsymbol\beta(t)$ is a line.

\item If the lines are coincident,  we can assume by Corollary \ref{cor:stdconic} that $J_0$ is the $y$ axis; $\det(DF(x,y))$ can be assumed to be $x^2$.
This turns out to be case 7b in the statement of the Theorem.
By equation (\ref{eq-jacdet}),
\begin{align}
2X_{01}&=1\label{eq:dlcoeff0}\\
4X_{02}&=0\label{eq:dlcoeff1}\\
2X_{12}&=0\label{eq:dlcoeff2}\\
2X_{04}+X_{31}&=0\label{eq:dlcoeff3}\\
2X_{32}+X_{14}&=0\label{eq:dlcoeff4}\\
X_{34}&=0\label{eq:dlcoeff5}.
\end{align}

First, (\ref{eq:dlcoeff0}, \ref{eq:dlcoeff1}, \ref{eq:dlcoeff2}) together with 
(aX.201) imply $a_2=0$, and together with (bX.201) imply $b_2=0$.
So $X_{2i}=0$ for all $i$.
(XX.0134) now reduces to $X_{04}X_{13}=0$, but 
together with (\ref{eq:dlcoeff3}) this forces
both $X_{04}$ and $X_{13}$ to equal zero.
Now, since $X_{04}$ and $X_{14}$ are both zero, and 
$X_{01}=1/2\ne 0$, (aX.401) implies $a_4=0$ and
(bX.401) implies $b_4=0$.
This forces $\boldsymbol\alpha(t)$ to be a point.
So when $J_0$ is a double line, $J_1$ can only be a point.
\end{enumerate}

\item $J_0$ is a simple line. We can assume by Corollary \ref{cor:stdconic} that $J_0$ is the $y$ axis; $\det(DF(x,y))$ can be assumed to be $x$.
This turns out to include cases 8a, 8b and 8c in the Theorem statement.
By equation (\ref{eq-jacdet}),
\begin{align}
2X_{01}&=0\label{eq:lcoeff0}\\
4X_{02}&=0\label{eq:lcoeff1}\\
2X_{12}&=0\label{eq:lcoeff2}\\
2X_{04}+X_{31}&=1\label{eq:lcoeff3}\\
2X_{32}+X_{14}&=0\label{eq:lcoeff4}\\
X_{34}&=0\label{eq:lcoeff5}.
\end{align}
Parametrize $J_0$ by $\{(0,t)\}$.
Then $J_1$ is parametrized by 
$\{ \boldsymbol\alpha(t) \}$
where $\boldsymbol\alpha(t) = 
(a_2 t^2 + a_4 t, b_2 t^2+ b_4 t)$.
By Lemma \ref{lemma:quadimage}, $J_1$ can
be a point, line, parabola, or ray.
Our examples 8a, 8b and 8c show that the first three are possible, so all we need
to do is exclude the case of a ray.
By comparing the parametrization of $J_1$ with
Lemma \ref{lemma:quadimage}, we see that the
quantity $\Gamma$ from the lemma is exactly $X_{24}$.
Thus, the condition for a ray is that 
$X_{24}=0$ where at least one of $a_2$ or $b_2$ is nonzero.  We will show that $X_{24}=0$ implies both
$a_2=0$ and $b_2=0$.  This will be done in three subitems: $a_4\ne 0$, $b_4 \ne 0$, and $a_4=0=b_4$.
\begin{itemize}
\item $a_4 \ne 0$.
First, (aX.423) implies $X_{23}=0$.
Then (\ref{eq:lcoeff4}) implies $X_{14}=0$.
Then (aX.413) implies $X_{13}=0$, after which
(aX.402) implies $a_2 X_{40}=0$. $X_{40}$ cannot equal zero because, along with $X_{13}=0$, 
it would contradict (\ref{eq:lcoeff3}).
Therefore $a_2=0$.
But now we would have $X_{24}=0$ and $a_2=0$ along
with $a_4 \ne 0$.  This forces $b_2=0$.
But $J_1$ cannot be a ray if both $a_2$ and $b_2$ are zero.
\item $b_4 \ne 0$.
This case leads to both $a_2$ and $b_2$ vanishing exactly in the previous case by interchanging $a$ and $b$.
\item $a_4=0=b_4$.
This forces $X_{4i}=0$ for all $i$.
Now (\ref{eq:lcoeff4}) implies $X_{23}=0$ and
(\ref{eq:lcoeff3}) implies $X_{31}=1$.
Now (aX.312) implies $a_2=0$, and (bX.312) implies $b_2=0$.
\end{itemize}
These three subitems show that a ray is impossible, so $J_1$ can only be a point, line, or parabola.

\item $J_0$ is all of $\mR^2$.
The proof we use in this case is completely different, since $J_0$ is not one-dimensional, and therefore has no curve parametrization to exploit.
On the other hand, $\det(DF(x,y))$ vanishes identically, so equation (\ref{eq-jacdet}) implies
\begin{align}
2X_{01}&=0\label{eq:r2coeff0}\\
4X_{02}&=0\label{eq:r2coeff1}\\
2X_{12}&=0\label{eq:r2coeff2}\\
2X_{04}+X_{31}&=0\label{eq:r2coeff3}\\
2X_{32}+X_{14}&=0\label{eq:r2coeff4}\\
X_{34}&=0\label{eq:r2coeff5}.
\end{align}
These six equations could be used as in the eight cases above to obtain the result, but our version is several pages long.  Instead, we quote a completed calculation from work in progress \cite{PKO}, 
where the general twelve-parameter family of eq.~(\ref{eq-genquad}) is equivalent, in the sense of Lemma~\ref{lemma:equivalence}, to the seven-parameter family
$(x^2+ a_2y^2+a_3x+a_4y, b_1 xy+ b_2y^2+b_3x+b_4y)$.
This can be done by choosing $h$ and $k$ in Lemma~\ref{lemma:equivalence} to make $(1,0)$ map to $(1,0)$, and to make sure $\max||F(\cos(\theta), \sin(\theta))||=1$.
($h$ is a rotation; $k$ is a rotation composed with a rescaling.)
This means we can assume $a_1=0$ and $b_0=0$.
Eqs.~(\ref{eq:r2coeff0}, \ref{eq:r2coeff1}, \ref{eq:r2coeff3}) then imply $b_1, b_2$ and $b_4$ vanish, and eq.~(\ref{eq:r2coeff4}) implies $a_2b_3=0$.
\begin{itemize}
\item If $b_3=0$, then $F(x,y)=(x^2+a_2 y^2 + a_3x+a_4y, 0)$. If $a_2<0$, $J_1$ is the image of the plane which covers the whole $x$-axis; if $a_2 \ge 0$, $J_1$ only covers a ray on the $x$-axis extending to $+\infty$. Together, $J_1$ is either a line or a ray.
\item If $b_3 \ne 0$, then $a_2=0$, and eq.~(\ref{eq:r2coeff5}) implies $a_4=0$.
So $F(x,y)=(x^2 + a_3x, b_3x)$, and $J_1$ is a parabola.
\end{itemize}

Thus the only cases that occur are $J_1$ as a line, ray or parabola.
\end{enumerate}
\end{proof}

Note: The proofs of the classification theorem above actually give stronger results than just the classification of the sets $J_0$ and $J_1$.
In addition, an inspection of the parametrizations and their images, illustrated in Fig.~\ref{fig-surfs}, gives pointwise information about the maps the from $J_0$ onto $J_1$.
Specifically, the maps restricted to $J_0$ onto $J_1$ are bijections in cases 1 (trivially), 2, 3, 4, 6, 8b and 8c. 
In 5a, one of the branches of the $J_0$ hyperbola maps bijectively to its image parabola; 
the other branch maps to a ray: the intersection of  the two lines in $J_0$ mapping to the endpoint of the ray, and the rest of the branch `folding' to map 2-1 to the rest of the ray.
In 5b both branches of the $J_0$ hyperbola map to the image rays as in case 5a.  The intersection of the lines in $J_0$ maps to the (common) vertex of the rays.
In case 7a, one line in $J_0$ maps bijectively to its image line; all points on the other line in $J_0$ map to the point in $J_1$.
In cases 7b and 8a the whole $J_0$ line maps to the $J_1$ point.
In case 9 of the Theorem (not pictured in Fig.~\ref{fig-surfs}), $J_0$ is the whole plane, and $J_1$ is either a line, parabola or ray.
Preimages of points in $J_1$ are either a single unbounded curve or a pair of unbounded curves.
In our case 9a example, the preimage of the origin is a pair of intersecting lines; all other points on the $J_1$ line have hyperbolas as preimages, with both branches of the hyperbola mapping to the same point.
In the case 9b example, the preimages of the $J_1$ ray vertex is the $x$-axis; the preimages of all other points on the ray are pairs of vertical lines of the form $x=\pm c$.
In example 9c, each point on the $J_1$ parabola has a single vertical line as its preimage.

\section{Discussion}
\label{sec:discussion}

One can view the classification of the singular set for quadratic maps as a map from the $12$-parameter quadratic map coefficient space (eq.~(\ref{eq-genquad})) to the $6$-dimensional coefficient space of the corresponding Jacobian determinant (eq.~(\ref{eq-jacdet})).
We note that this map is not onto.
In particular, there are conic sections corresponding to empty singular sets which are
not possible as Jacobian determinants of quadratic maps of the plane.
For example, $x^2+y^2+1$ cannot be realized as a Jacobian determinant of a quadratic map of the plane.
This can be seen by mimicking the proof of Case 3 in Theorem \ref{th-J0J1}, where we assumed $\det(DF(x,y))=x^2+y^2-1$.  Identities (a)-(e) remain
unchanged, but identity (f) becomes $X_{13}^2+X_{14}^2=-1$, which is impossible.
Similarly, $x^2+1$ cannot be realized, as can be seen by mimicking the proof of case 6 in Theorem \ref{th-J0J1}, where we assumed $\det(DF(x,y))=x^2-1$.
The implication that $X_{13}^2=1$ in case 6 becomes $X_{13}^2=-1$, which is also impossible.
An interesting consequence is that the only quadratic maps with empty critical set are those with constant nonzero Jacobian determinant, which include the Henon maps, as well as other maps which are homeomorphisms, but not conjugate to a Henon map \cite{NienNF}.
Work in progress \cite{PKO} treats this issue in more detail by addressing the geometry of the quadratic coefficient space with respect to the $J_0$-$J_1$ classification used in this paper; a consequence is identification of the codimension of each of the cases enumerated in Theorem \ref{th-J0J1}.
Other future work includes identifying normal forms, using topological equivalence rather than map equivalence, for each of the classes identified in the current paper and studying the dynamics of the corresponding families.
This would extend the program initiated in \cite{Nienthesis} for maps with critical sets which are either ellipses or points.

We note that generalizations of the approach in this paper could be applied to cubic maps of the plane, assuming only one component is cubic, while the other component is linear.  This also leads to conic sections for $J_0$, but new possibilities for $J_1$ beyond those enumerated in Theorem \ref{th-J0J1} for quadratic maps.  
See \cite{Nienthesis} for cubic examples and some generalizations to higher dimensional maps.

\section{Summary}
\label{sec:summary}

This paper is a complete classification of the critical sets and their images for quadratic maps of the plane.
Although our ultimate goal is a complete classification of the dynamics of quadratic maps,
the results in this paper are only a small step in that direction since we have studied a single iterate rather than the long term behavior under iteration.
Nonetheless, this global singularity theory approach provides a coarse classification of quadratic maps which we believe is a useful step on the way to understanding the full dynamical behavior of this family.

\section{Appendix}

This Appendix includes the proofs of the two generic cases of Theorem \ref{th-J0J1} where $J_0$ is (3) an ellipse, or (4) a hyperbola.  The proofs also follow from results in \cite{DGRRV}. Our proofs are included for completeness and because they differ from those in \cite{DGRRV}.

\begin{enumerate}
\setcounter{enumi}{2}
\item $J_0$ is an ellipse. As suggested in the statement just following Cor.~\ref{cor:stdconic}, we can assume
$\det(DF(x,y))=x^2+y^2-1$.

{\it Showing that $J_1$ is a closed curve with exactly three singular points.}
First parametrize $J_0$ by $\{(\cos(t), \sin(t))|t \in \mR \}$.
$J_0$ and $J_1$ are clearly a closed curves.
By eq.~(\ref{eq-genquad}), $J_1=F(J_0)$ is parametrized by
$\boldsymbol\alpha(t)=(\alpha_1(t), \alpha_2(t))=(a_0 \cos^2(t) + a_1 \cos(t)\sin(t) + a_2 \sin^2(t) + a_3 \cos(t) + a_4 \sin(t), 
b_0 \cos^2(t) + b_1 \cos(t)\sin(t) + b_2 \sin^2(t) + b_3 \cos(t) + b_4 \sin(t))$.

By the coefficients of $\det(DF(x,y))$ in equation (\ref{eq-jacdet}),
\begin{align}
2X_{01}&=1\label{eq:ecoeff0}\\
4X_{02}&=0\label{eq:ecoeff1}\\
2X_{12}&=1\label{eq:ecoeff2}\\
2X_{04}+X_{31}&=0\label{eq:ecoeff3}\\
2X_{32}+X_{14}&=0\label{eq:ecoeff4}\\
X_{34}&=-1\label{eq:ecoeff5}.
\end{align}

These six equations and the identities listed below lead to
\begin{enumerate}
\item $a_2=-a_0$, by (aX.201)
\item $a_3=-2(a_0X_{13}+a_1X_{30})$, by (aX.301)
\item $a_4=-2(a_0X_{14}+a_1X_{40})$, by (aX.401)
\item $X_{23}=-X_{03}=\frac{1}{2}X_{14}$, by (XX.0123) and (\ref{eq:ecoeff4})
\item $X_{24}=-X_{04}=\frac{1}{2}X_{31}$, by (XX.0124) and (\ref{eq:ecoeff5})
\item $X_{13}^2+X_{14}^2 = 1$, by (XX.0134)
\end{enumerate}

The last equation allows us to define $\phi$ by  
$X_{13}=-\cos(\phi), X_{14}=\sin(\phi)$.

Now we can rewrite $\alpha_1'(t)$ as

\begin{align}
\alpha_1'(t)&=(a_2- a_0) 2\cos(t)\sin(t) + a_1 \cos^2(t) -a_1 \sin^2(t) + 2 a_2 \sin(t)\cos(t)\nonumber\\
& \hskip .3in - a_3 \sin(t) + a_4 \cos(t) 
\rm{\ (by\ differentiation\ of\ } \boldsymbol\alpha(t))\nonumber \\
&=-2a_0 \sin(2t)+a_1\cos(2t)-a_3\sin(t)+a_4\cos(t)
\nonumber\\
&\hskip .3in \rm{\ (by\ double\ angle\ identities\ and\ identity\ (a))}\nonumber\\
&=-2a_0\sin(2t)+a_1\cos(2t)-(-2)(a_0X_{13}+a_1X_{30})\sin(t)\nonumber\\
&\hskip .3in +(-2)(a_0X_{14}+a_1X_{40})\cos(t)
\nonumber\\
&\hskip .3in \rm{\ (by\ identities\ (b)\ and\ (c))}\nonumber\\
&=-2a_0(\sin(2t)+\cos(\phi)\sin(t)+\sin(\phi)\cos(t))
\nonumber\\
&\hskip .3in +a_1(\cos(2t)+\sin(\phi)\sin(t)-\cos(\phi)\cos(t))\nonumber\\
&\hskip .3in \rm{\ (by\ identities\ (d)\ and\ (e)\ and\ definition\ of\ }\phi)\nonumber\\
&=-2a_0(\sin(2t)+\sin(t+\phi))+a_1(\cos(2t)-\cos(t+\phi))\nonumber\\
&\hskip .3in \rm{\ (by\ sum\ angle\ identities)}\nonumber\\
&=-2a_0(2\sin(\frac{3t+\phi}{2}) \cos(\frac{\phi-t}{2})) + a_1(2\sin(\frac{3t+\phi}{2}) \sin(\frac{\phi-t}{2}))\nonumber\\
&\hskip .3in ({\rm since\ }\sin(A+B)+\sin(A-B)=2\sin A\cos B \rm{\ and\ }\nonumber \\
&\hskip .3in \cos(A+B)-\cos(A-B)=-2\sin A\sin B) \nonumber\\
&=2\sin(\frac{3t+\phi}{2})(-2a_0 \cos(\frac{\phi-t}{2}) + a_1 \sin(\frac{\phi-t}{2}))
\label{eq:a1prime}
\end{align}

Similarly,
\begin{align}
\alpha_2'(t))&=2\sin(\frac{3t+\phi}{2})(-2b_0 \cos(\frac{\phi-t}{2}) + b_1 \sin(\frac{\phi-t}{2}))
\label{eq:a2prime}
\end{align}

Define $A1(t)$ and $A2(t)$ via eqs.~(\ref{eq:a1prime}) and (\ref{eq:a2prime}) so that 
\[(\alpha_1'(t), \alpha_2'(t))=2\sin(\frac{3t+\phi}{2})(A1(t), A2(t)).\]

It is clear that $(\alpha_1'(t), \alpha_2'(t))=(0,0)$ at the three $t$ values defined by $3t+\phi=0 \mod 2\pi$.
It turns out that these are the only solutions since $A1(t)$ and $A2(t)$ cannot simultaneously be zero for any $t$ value.
This can be seen because if there were a simultaneous zero,  
$b_0 A1(t) - a_0 A2(t) = X_{10} \sin(\frac{\phi-t}{2})=-\frac{1}{2} \sin(\frac{\phi-t}{2})$, and
$b_1 A1(t) - a_1 A2(t) = 2 X_{10} \cos(\frac{\phi-t}{2})=-\cos(\frac{\phi-t}{2})$ would simultaneously be zero, which is impossible.
Thus 
$(\alpha_1'(t), \alpha_2'(t))=(0,0)$ has exactly three solutions.

{\it Showing these three solutions are, in fact, nondegenerate cusps.}
By Lemma \ref{lemma:cusp}, this requires verifying the nondegeneracy condition
$\alpha_1''(t) \alpha_2'''(t) - \alpha_2''(t) \alpha_1'''(t) \ne 0$ at the three zeros.
Differentiating eqs.~(\ref{eq:a1prime}) and (\ref{eq:a2prime}) twice, this quantity can be shown to equal
$2X_{01}(7+2\cos(3t+\phi))$ which clearly cannot equal zero at any $t$, so does not equal zero at the $t$ values corresponding to the three cusps.

\item $J_0$ is a hyperbola.
By Corollary \ref{cor:stdconic}, we can assume that $J_0$ is the special hyperbola $xy=1$, and $\det(DF(x,y))=xy-1$.
By equation (\ref{eq-jacdet}),
\begin{align}
2X_{01}&=0\label{eq:hcoeff0}\\
4X_{02}&=1\label{eq:hcoeff1}\\
2X_{12}&=0\label{eq:hcoeff2}\\
2X_{04}+X_{31}&=0\label{eq:hcoeff3}\\
2X_{32}+X_{14}&=0\label{eq:hcoeff4}\\
X_{34}&=-1\label{eq:hcoeff5}.
\end{align}

Parametrize $J_0$ by $\{(t, 1/t) | t \ne 0 \}$.
Then $J_1$ is parametrized by 
$\{(\alpha_1(t), \alpha_2(t)) | t \ne 0 \}$
where $(\alpha_1(t), \alpha_2(t)) = (a_0 t^2 + a_1 + a_2 t^{-2}+a_3 t + a_4 t^{-1}, 
b_0 t^2 + b_1 + b_2 t^{-2}+b_3 t + b_4 t^{-1})$.
We will show there is exactly one solution to
$(\alpha_1'(t), \alpha_2'(t))=(0,0)$.
We will first show that there is only one possible solution, then that this is, in fact, a solution, and finally, that the solution is a nondegenerate cusp on $J_1$.

{\it Showing there is only one possible singular point on $J_1$.}
After differentiating $(\alpha_1(t), \alpha_2(t))$ and multiplying through by $t^3$, we see that any solution satisfies:
\begin{align}
A1(t) \equiv t^3\alpha_1'(t)=-2a_2-a_4t+a_3t^3+2a_0 t^4&=0\label{eq:ha1prime}\\
A2(t) \equiv t^3\alpha_2'(t)=-2b_2-b_4t+b_3t^3+2b_0 t^4&=0\label{eq:ha2prime}
\end{align}

By $b_0$ (\ref{eq:ha1prime}) - $a_0$ (\ref{eq:ha2prime}), we eliminate $t^4$ and obtain $-2X_{20}-X_{40}t+X_{30}t^3=0$.
It can be shown that $X_{40}=0$ as follows: use eq.~(aX.201)
with (\ref{eq:hcoeff0}, \ref{eq:hcoeff1}, \ref{eq:hcoeff2}) to obtain $a_1=0$;
similarly use eq.~(bX.201) to obtain $b_1=0$;
so $X_{1i}=0$ for all $i$;
now eq.~(\ref{eq:hcoeff3}) implies $X_{40}=0$.
Now eq.~(XX.0234): $X_{02}X_{34}-X_{03}X_{24}+X_{04}X_{23}=0$ implies
$\frac{1}{4}(-1)-X_{03}X_{24}+0X_{23}=0$, so  $X_{30}\ne 0$.
Therefore, $t=(-\frac{1}{2X_{30}})^{1/3}\equiv T$.

{\it Showing this possible singular point on $J_1$ satisfies $(\alpha_1'(T), \alpha_2'(T))=(0,0)$.}
Eq.~(\ref{eq:ha1prime}) implies

\begin{align}
\alpha_1'(T)&= (-2a_2(1/T^3)+a_3)+T(-a_4(1/T^3)+2a_0)\nonumber \\
&=[-a_2(-2X_{30})+a_3]+T[-a_4(-2X_{30})+2a_0]
\end{align}
Both expressions in the square brackets vanish, the first by (aX.302) and the second by (aX.403).
Similarly, by interchanging $a_i$ and $b_i$, $\alpha_2'(T)=0$.

{\it Showing $\boldsymbol\alpha(T)$ is a nondegenerate cusp.}
By Lemma \ref{lemma:cusp}, we must show 
$\alpha_1''(T) \alpha_2'''(T) - \alpha_1'''(T) \alpha_2''(T) \ne 0$.
Since $\boldsymbol\alpha'(T)=(0,0)$, eqs.~(\ref{eq:ha1prime}) and (\ref{eq:ha2prime})
can be used to show 
the nondegeneracy condition is equivalent to 
$A1'(T) A2''(T) - A1''(T) A2'(T) \ne 0$:
\begin{align}
A1'(T)& A2''(T) - A1''(T) A2'(T)\nonumber\\
&=(-a_4 + 3a_3T^2+8a_0T^3)(6b_3 T + 24b_0 T^2)\nonumber\\
&\hskip .3in -(-b_4 + 3b_3T^2+8b_0T^3)(6a_3 T + 24a_0 T^2)\nonumber\\
&=6X_{34}T+72X_{30}T^4+48X_{03}T^4\nonumber\\
&=6T(X_{34}+4X_{30} T^3)\nonumber\\
&=6T(-1 -2) \rm{\ by\ eq.\ (\ref{eq:hcoeff5})\ and\  definition\ of\ }T \nonumber\\
&=-18T \ne 0.
\end{align}

\end{enumerate}

\section*{Acknowledgments} BBP acknowledges useful discussions with Bernd Krauskopf and Hinke M. Osinga while preparing a related manuscript \cite{PKO} while on sabbatical leave at the University of Auckland during the first four months of 2013.  RPM acknowledges support from NSF Grant DMS-0940366.



\end{document}